\documentclass[12pt,a4paper]{article}%

\usepackage{amsmath}
\usepackage{amstext}
\usepackage{graphicx}
\usepackage{enumerate}
\usepackage{amsfonts}
\usepackage{amssymb}%

\providecommand{\U}[1]{\protect\rule{.1in}{.1in}}

\newtheorem{theorem}{Theorem}

\newtheorem{example}[theorem]{Example}
\newtheorem{lemma}[theorem]{Lemma}

\newtheorem{proposition}[theorem]{Proposition}
\newtheorem{remark}[theorem]{Remark}

\newenvironment{proof}[1][Proof]{\noindent\textbf{#1.} }{\ \rule{0.5em}{0.5em}}

\begin{document}

\title{On the recurrence and robust properties of Lorenz'63 model }
\author{M. Gianfelice\thanks{Partially supported by GREFI-MEFI and PEPS
\emph{Mathematical Methods of Climate Models}. S.V. wish to thank G.
Cristadoro, W. Bahsoun, R. Aimino and I. Melbourne for useful discussions and comments.}, F.
Maimone\thanks{Partially supported by PEPS \emph{Mathematical Methods of
Climate Models}}, V. Pelino$^{\dagger}$, S. Vaienti$^{\ast}$}
\maketitle

\begin{abstract}
Lie-Poisson structure of the Lorenz'63 system gives a physical insight on its
dynamical and statistical behavior considering the evolution of the associated
Casimir functions. We study the invariant density and other recurrence
features of a Markov expanding Lorenz-like map of the interval arising in the
analysis of the predictability of the extreme values reached by particular
physical observables evolving in time under the Lorenz'63 dynamics with the
classical set of parameters. Moreover, we prove the statistical stability of
such an invariant measure. This will allow us to further characterize the SRB
measure of the system.

\end{abstract}
\tableofcontents



\section{Introduction}

In 1963 E. Lorenz by a drastic truncation of fluid-dynamics equations
governing the atmospheric motion obtained a system of ODE which he proposed as
a crude yet non trivial model of thermal convection of the atmosphere
\cite{L}. As a matter of fact, the Lorenz model is today understood as a basic
toy-model for the evolution of Earth atmosphere's regimes, like zonal or
blocked circulation or climate regimes (e.g. warm and cold), in which dynamics
is described by equilibrium states \cite{CMP}, \cite{Se}. In his work Lorenz
showed such system to exhibit, for a large set of parameters values, a
peculiar chaotic behavior, that is exponential sensitivity to perturbations of
initial conditions and the existence of a global attracting set for the flow
nowadays called generalized nontrivial hyperbolic attractor. Although there
exists an extensive literature on the subject, we refer the reader to
\cite{Sp} for a rather comprehensive overview on this problem and to \cite{V}
for a recent account on the progress made on the rigorous analysis of the
Lorenz '63 ODE\ system and the relationship between this and its more abstract
counterpart, the geometric Lorenz model, introduced in the second half of the
seventies (\cite{G}, \cite{ABS} and \cite{GW}) to describe the geometrical
features a dynamical system should posses in order to exhibit the same
asymptotic behavior as the Lorenz one. An affirmative answer to the
long-standing question whether the original Lorenz '63 flow fits or not the
description of the one modeled by the geometric Lorenz, which means it
supports a robust singular hyperbolic (strange) attractor, has been given by
W. Tucker in \cite{T} by means of a computer assisted proof. As a byproduct,
Tucker also proved that Lorenz flow admits a unique SRB measure supported on
the strange attractor. Further results about the characterization of the set
of geometric Lorenz-like maps are given in \cite{LM}.

In 2000 it has been emphasized that Lorenz '63 model and the Kolmogorov one,
considered as a low-order approximation of the Navier-Stokes equations, belong
to a particular class of dynamical systems, named Kolmogorov-Lorenz systems
\cite{PP1}, whose vector field admits a representation as a sum of a
Hamiltonian $SO\left(  3\right)  $-invariant field, a dissipative linear field
and constant forcing field (see also \cite{PP2} for an extension of this
analysis to the Lorenz '84 model). Moreover, they proved that the chaotic
behavior of these models relies on the interplay between dissipation and forcing.

More specifically, and more recently, in \cite{PM} it has been shown that the
effect of the dissipative and forcing terms appearing in the previously
described decomposition of the Lorenz '63 vector field, with the classical set
of parameters, is to induce chaotic oscillations in the time evolution of the
first integrals of the Hamiltonian system associated to the Lorenz '63 model,
namely the Hamiltonian and the Casimir function for the (+) Lie-Poisson
brackets associated the $so\left(  3\right)  $ algebra \cite{MR}, which
represents the angular momentum of a free rigid body in the Kolmogorov-Lorenz
representation of geofluid dynamics introduced in \cite{PP1}. In particular,
it has been shown that two subsequent oscillation peaks in the plot of the
Casimir function $C$ as a function of time are related by a map $\Phi$ of the
interval similar to the one originally computed by Lorenz in \cite{L}
depicting the functional dependence of two subsequent maximum values assumed
by the third coordinate of the flow as a function of time. We remark that,
being $C$ the square norm of the flow, the similarity of the plots of this two
maps is therefore not surprising. In \cite{PM}, the recurrence properties of
$\Phi$ are also studied, which allows to characterize the trajectories of the
system through the number of revolution they perform around the unstable point
lying on one side of the plane $x+y=0$ when the initial condition is chosen on
the opposite side (\cite{PM} figg. 2,10 and 11).

In our paper we clarify what stated in \cite{PM} by giving an account in the
first section of the rigid body formulation of the Lorenz '63 model and
constructing, in the next section, a Markov expanding Lorenz-like map $T$ of
the interval being the reduction to $\left[  0,1\right]  $ of $\Phi.$ Both
maps are in fact derived throughout the Poincar\'{e} map associated to the
surface in the configuration space of the system corresponding to the set of
maxima reached by the Casimir function during its time evolution. Hence, we
will study the invariant measure under the dynamics defined by $T$
characterizing its density and consequently the SRB measure of the system.
Furthermore, we analyse the recurrence properties of the dynamics induced by
$T$ clarifying more rigorously what stated in Section IV B of \cite{PM}.

We will also perturb the system by adding an extra forcing term which will
eventually cause the system to loose its symmetry under the involution
$R:\left(  x,y,z\right)  \rightarrow\left(  -x,-y,z\right)  $ of
$\mathbb{R}^{3}.$ Due to the robustness of the attractor, i.e. persistence
under perturbations of the parameters, proved in \cite{T}, maps analogous to
$T$ can be defined and studied, and their statistical properties analysed as
in the unperurbed case. Therefore, such perturbation of the Lorenz '63 field
will only have the effect to induce a change in the statistics of the
invariant measure for the system, the SRB measure. We will prove that this
change could be detected by looking at the deviation of the invariant density
of the perturbed map with respect to the unperturbed one. Such result would
confirm what has been empirically shown in \cite{CMP} about the impact of
anthropogenic forcing to climate dynamics of the northern hemisphere.

We also believe that this analysis could also be pursued in the case of more
general $N$-dimensional models such as those introduced by Zeitlin in \cite{Z}
to approximate in the limit of $N$ tending to infinity the dynamics of the
atmosphere in absence of dissipation and forcing.

We will present elsewhere our contributions in these directions; here we prove
the first non-trivial result about the statistical stability of the invariant
measure for $T.$ The technique we propose is new and we believe it could be
applied as well for other maps with some sort of criticalities.

We remark that, in particular, the distribution of the return times of a
measurable subset of $[0,1],$ which can be derived directly from the invariant
measure of $T,$ could be useful in studying the statistics of extreme
meteorological events.

\section{Return Lorenz-like maps}

\subsection{Rigid body formulation of the Lorenz '63 model\label{KLS}}

It can be shown (\cite{PP1}) that the Lorenz '63 ODE system \cite{L},
\begin{equation}
\left\{
\begin{array}
[c]{l}%
\dot{x}_{1}=-\sigma x_{1}+\sigma x_{2}\\
\dot{x}_{2}=-x_{1}x_{3}+\rho x_{1}-x_{2}\\
\dot{x}_{3}=x_{1}x_{2}-\beta x_{3}%
\end{array}
\right.  \label{l}%
\end{equation}
can be mapped, through the change of variables
\begin{equation}
\left\{
\begin{array}
[c]{l}%
u_{1}=x_{1}\\
u_{2}=x_{2}\\
u_{3}=x_{3}-\left(  \rho+\sigma\right)
\end{array}
\right.  \ ,
\end{equation}
to the ODE system
\begin{equation}
\left\{
\begin{array}
[c]{l}%
\dot{u}_{1}=-\sigma u_{1}+\sigma u_{2}\\
\dot{u}_{2}=-u_{1}u_{3}-\sigma u_{1}-u_{2}\\
\dot{u}_{3}=u_{1}u_{2}-\beta u_{3}-\beta\left(  \rho+\sigma\right)
\end{array}
\right.  \label{l1}%
\end{equation}
representing the evolution of a Hamiltonian system whose configuration space
is the $SO\left(  3\right)  $ group, subject to dissipation and to a constant
forcing. That is, denoting by
\begin{equation}
\{F,G\}:=\omega_{+}^{2}\left(  ad_{\nabla F}^{\ast}x,ad_{\nabla G}^{\ast
}x\right)  =x\cdot\nabla F\times\nabla G \label{LPb}%
\end{equation}
the Lie-Poisson brackets associated to the symplectic 2-form $\omega_{+}^{2}$
defined on the cotangent bundle of $SO\left(  3\right)  $ \cite{MR},
(\ref{l1}) reads
\begin{equation}
\dot{u}_{i}=\{u_{i},H\}-\left(  \Lambda u\right)  _{i}+f_{i}\ ,\qquad i=1,2,3,
\label{lpp}%
\end{equation}
where:

\begin{itemize}
\item
\begin{equation}
H\left(  u\right)  :=\frac{1}{2}u\cdot\Omega u+h\cdot u
\end{equation}
is the Hamiltonian of a rigid body whose kinetical term is given by the matrix
$\Omega:=diag\left(  2,1,1\right)  ,$ while $h:=\left(  0,0,-\sigma\right)  $
is an axial torque;

\item $\Lambda:=diag\left(  \sigma,1,\beta\right)  $ is the dissipation matrix;

\item $f:=\left(  0,0,-\beta\left(  \rho+\sigma\right)  \right)  $ is a
forcing term.
\end{itemize}

This representation allows to study the Lorenz system as a perturbation of the
Hamiltonian system
\begin{equation}
v_{i}\left(  u\right)  :=\{u_{i},H\},\qquad i=1,2,3, \label{lppi}%
\end{equation}
admitting, as in the case of a rigid body with a fixed point, two independent
first integrals $H$ and the Casimir function $C$ for the Poisson brackets
(\ref{LPb}) \cite{MR}. In fact, when rewritten in this form, it follows
straightforwardly that the system is non chaotic for $\sigma=0$ \cite{PP1}
while, for $\sigma\neq0,$ the values of $C$ and $H$ undergo chaotic
oscillations \cite{PM}.

Moreover, when passing to the representation (\ref{lpp}), the symmetries of
the system are preserved as well as other features such as the invariance of
the $x_{3}\ \left(  u_{3}\right)  $ axis and the direction of rotation of the
trajectories about this axis. The critical points of the velocity field of the
system are then
\begin{equation}
c_{1}:=\left(  \sqrt{\beta\left(  \rho-1\right)  },\sqrt{\beta\left(
\rho-1\right)  },-\left(  \sigma+1\right)  \right)  ,\ c_{2}:=\left(
-x_{1}\left(  c_{1}\right)  ,-x_{2}\left(  c_{1}\right)  ,x_{3}\left(
c_{1}\right)  \right)  \ ,
\end{equation}
and $c_{0}:=\left(  0,0,-\left(  \rho+\sigma\right)  \right)  .$

We also remark that (\ref{lpp}) can be rewritten in the form
\begin{equation}
\dot{u}=v-w\ , \label{decu'}%
\end{equation}
where $v$ is the divergence free field (\ref{lppi}) and
\begin{equation}
\mathbb{R}^{3}\ni u\longmapsto w\left(  u\right)  :=\Lambda u-f=\nabla
K\left(  u\right)  \in\mathbb{R}^{3}%
\end{equation}
with
\begin{equation}
K\left(  u\right)  :=\frac{1}{2}u\cdot\Lambda u-f\cdot u
\end{equation}
a convex function on $\mathbb{R}^{3}.$ Notice that the fields $v$ and $w$ are
orthogonal in $L^{2}\left(  r\mathcal{B};\mathbb{R}^{3}\right)  ,$ where
$B:=\{u\in\mathbb{R}^{3}:\left\Vert u\right\Vert \leq1\}$ is the unitary ball
in $\mathbb{R}^{3}$ and $rB$ denotes the ball of radius $r.$

The decomposition of the velocity field as the sum of a divergence free field
ad a gradient one, together with the appearance in the Hamiltonian description
of the flow of the Lie-Poisson brackets (\ref{LPb}) in the space reference
frame of the rigid body, i.e. right translation on $SO\left(  3\right)  ,$ is
standard in fluid dynamics \cite{A}, \cite{MR} and can be seen as another
source of analogy between the Lorenz '63 model and Navier-Stokes equations
\cite{PP1}, \cite{FJKT}.

\subsection{The return map on the set of maxima of the Casimir function}

If $u_{0}$ is any non stationary point for the field (\ref{l1}) such that
$\left\Vert u_{0}\right\Vert \leq\frac{\left\Vert f\right\Vert }{\sqrt
{\lambda_{\Lambda}}},$ with $\lambda_{\Lambda}:=\min\{t\in spec\Lambda\},$ and
$C\left(  t\right)  :=\left\Vert u\left(  t,u_{0}\right)  \right\Vert ^{2},$
let
\begin{equation}
m:=\inf_{t>0}C\left(  t\right)  \ ;\ M:=\sup_{t>0}C\left(  t\right)  \ .
\end{equation}
Clearly, $m\geq0$ since $C\geq0.$ Moreover, $M<\infty$ since it has been shown
in \cite{PP1} that $C\left(  t\right)  \leq\frac{\left\Vert f\right\Vert
}{\sqrt{\lambda_{\Lambda}}}$ where, for our choice of parameters,
\begin{equation}
\frac{\left\Vert f\right\Vert }{\sqrt{\lambda_{\Lambda}}}=\left\Vert
f\right\Vert =\beta\left(  \rho+\sigma\right)  \ .
\end{equation}

To construct the function which links two subsequent relative maximum values
of $C\left(  t\right)  $ we proceed as follows:

\begin{itemize}
\item first we identify the manifold $\Sigma$ in the configuration space of
the system corresponding to the relative maxima of $C\left(  t\right)  ,$

\item then we construct a map of the interval $\left[  0,1\right]  $ to itself
as a function of the map of the interval $\left[  m,M\right]  $ of the
possible values of $C\left(  t\right)  $ in itself, which can be defined
through the Poincar\'{e} map of this manifold.
\end{itemize}

The existence of the aforementioned Poincar\'{e} map follows from the
existence of the return map computed by Tucker in \cite{T}.

Throughout the paper we will assume $\sigma=10,\ \rho=28,\ \beta=\frac{8}{3}.
$

\subsubsection{Identification of $\Sigma$}

By (\ref{decu'}) we get
\begin{equation}
\dot{C}\left(  u\right)  =-2\left[  E\left(  u\right)  -\frac{\beta\left(
\rho+\sigma\right)  ^{2}}{4}\right]  \ . \label{C'-E}%
\end{equation}
with
\begin{equation}
E\left(  u\right)  :=\sigma u_{1}^{2}+u_{2}^{2}+\beta\left(  u_{3}%
+\frac{\left(  \rho+\sigma\right)  }{2}\right)  ^{2}\ .
\end{equation}
Therefore,
\begin{equation}
\mathcal{E}:=\left\{  u\in\mathbb{R}^{3}:\dot{C}\left(  u\right)  =0\right\}
=\left\{  u\in\mathbb{R}^{3}:E\left(  u\right)  =\frac{\beta\left(
\rho+\sigma\right)  ^{2}}{4}\right\}  \ ,
\end{equation}
as already noticed in \cite{PM}, is an ellipsoid intersecting the vertical
axis ($u_{3}$) in the origin and in $c_{0}.$ This also implies, $M=\rho
+\sigma.$ Clearly, $c_{1},c_{2}\in\mathcal{E}.$

Moreover, by (\ref{C'-E}) and (\ref{LPb})%
\begin{align}
\ddot{C}\left(  u\right)   &  =2\nabla E\cdot\left[  u\times\nabla H+\nabla
K\right]  \left(  u\right) \\
&  =4\left\{  \sigma^{2}u_{1}^{2}+u_{2}^{2}-\left[  \sigma\left(
\sigma-1\right)  +\left(  \beta-1\right)  \left(  u_{3}+\frac{\rho+\sigma}%
{2}\right)  +\frac{\rho+\sigma}{2}\right]  u_{1}u_{2}\right. \nonumber\\
&  \left.  +\beta^{2}\left(  u_{3}+\frac{\rho+\sigma}{2}\right)  ^{2}%
+\beta^{2}\frac{\rho+\sigma}{2}\left(  u_{3}+\frac{\rho+\sigma}{2}\right)
\right\}  \ .\nonumber
\end{align}
Let us set $z:=u_{3}+\frac{\rho+\sigma}{2},$ then
\begin{gather}
\mathcal{E}^{\prime}:=\left\{  u\in\mathbb{R}^{3}:\ddot{C}\left(  u\right)
=0\right\} \\
=\left\{  u\in\mathbb{R}^{3}:\sigma^{2}u_{1}^{2}+u_{2}^{2}-\left[
\sigma\left(  \sigma-1\right)  +\left(  \beta-1\right)  z+\frac{\rho+\sigma
}{2}\right]  u_{1}u_{2}\right. \nonumber\\
\left.  +\beta^{2}z\left(  z+\frac{\rho+\sigma}{2}\right)  =0\right\}
\ .\nonumber
\end{gather}
We remark that, denoting by $R$ the involution
\begin{equation}
\mathbb{R}^{3}\ni u=\left(  u_{1},u_{2},u_{3}\right)  \longmapsto Ru:=\left(
-u_{1},-u_{2},u_{3}\right)  \in\mathbb{R}^{3}\ ,
\end{equation}
leaving invariant the field $\dot{u},\ R\mathcal{E}=\mathcal{E}$ and
$R\mathcal{E}^{\prime}=\mathcal{E}^{\prime}.$

Consider the diffeomorphism
\begin{equation}
q_{i}=O\left(  z\right)  u_{i},\ i=1,2\ ;\ q_{3}=z
\end{equation}
such that for any fixed value of $z,\ O\left(  z\right)  $ is an orthogonal
matrix diagonalizing the symmetric quadratic form
\begin{equation}
\zeta\cdot A\left(  z\right)  \zeta:=\sigma^{2}\zeta_{1}^{2}+\zeta_{2}%
^{2}-\left[  \sigma\left(  \sigma-1\right)  +\left(  \beta-1\right)
z+\frac{\rho+\sigma}{2}\right]  \zeta_{1}\zeta_{2}\;,\qquad\zeta\in
\mathbb{R}^{2}\ ,
\end{equation}
namely, setting $A\left(  z\right)  =O^{t}\left(  z\right)  diag\left(
\lambda_{1}\left(  z\right)  ,\lambda_{2}\left(  z\right)  \right)  O\left(
z\right)  ,$%
\begin{equation}
u\cdot A\left(  z\right)  u=q\cdot O\left(  z\right)  A\left(  z\right)
O^{t}\left(  z\right)  q=\lambda_{1}\left(  z\right)  q_{1}^{2}+\lambda
_{2}\left(  z\right)  q_{2}^{2}\ .
\end{equation}
with
\begin{align}
\lambda_{1}\left(  z\right)   &  =\frac{\sigma^{2}+1+\sqrt{\left(  \sigma
^{2}-1\right)  ^{2}+\left[  \frac{\rho+\sigma}{2}+\sigma\left(  \sigma
-1\right)  +\left(  \beta-1\right)  z\right]  ^{2}}}{2}\ ,\\
\lambda_{2}\left(  z\right)   &  =\frac{\sigma^{2}+1-\sqrt{\left(  \sigma
^{2}-1\right)  ^{2}+\left[  \frac{\rho+\sigma}{2}+\sigma\left(  \sigma
-1\right)  +\left(  \beta-1\right)  z\right]  ^{2}}}{2}\ .
\end{align}
Under this change of variables
\begin{equation}
\mathcal{E}^{\prime}=\left\{  q\in\mathbb{R}^{3}:\lambda_{1}\left(
q_{3}\right)  q_{1}^{2}+\lambda_{2}\left(  q_{3}\right)  q_{2}^{2}+\beta
^{2}q_{3}\left(  q_{3}+\frac{\rho+\sigma}{2}\right)  =0\right\}  \ .
\end{equation}
Since $\lambda_{1}\left(  q_{3}\right)  $ is positive for any choice of the
parameters $\beta,\rho,\sigma$ and $q_{3},$ the equation giving the
intersection of $\mathcal{E}^{\prime}$ with the planes parallel to $q_{3}%
=0$\linebreak($u_{3}=-\frac{\left(  \rho+\sigma\right)  }{2}$) can have a
solution only if $\lambda_{2}\left(  q_{3}\right)  $ is negative, that is for
\begin{align}
q_{3}  &  >-\frac{\sigma\left(  \sigma-3\right)  +\frac{\rho+\sigma}{2}}%
{\beta-1}\ \Rightarrow\ u_{3}>-\frac{1}{\beta-1}\left[  \sigma\left(
\sigma-3\right)  +\beta\frac{\left(  \rho+\sigma\right)  }{2}\right]  \ ;\\
q_{3}  &  <-\frac{\sigma\left(  \sigma+1\right)  +\frac{\rho+\sigma}{2}}%
{\beta-1}\ \Rightarrow\ u_{3}<-\frac{1}{\beta-1}\left[  \sigma\left(
\sigma+1\right)  +\beta\frac{\left(  \rho+\sigma\right)  }{2}\right]  \ .
\end{align}
Therefore, for $q_{3}\neq0\ $($u_{3}\neq-\frac{\left(  \rho+\sigma\right)
}{2}$), these intersections are hyperbolas while, if $q_{3}=0$ or
$q_{3}=-\frac{\left(  \rho+\sigma\right)  }{2}\ $($u_{3}=-\left(  \rho
+\sigma\right)  $), from the definition of $\mathcal{E}^{\prime}$ we get the equations

\begin{itemize}
\item if $q_{3}=0,$%
\begin{equation}
\sigma^{2}u_{1}^{2}+u_{2}^{2}-\left[  \frac{\left(  \rho+\sigma\right)  }%
{2}+\sigma\left(  \sigma-1\right)  \right]  u_{1}u_{2}=0\ ;
\end{equation}

\item if $q_{3}=-\frac{\left(  \rho+\sigma\right)  }{2},$%
\begin{equation}
\sigma^{2}u_{1}^{2}+u_{2}^{2}-\left[  \sigma\left(  \sigma-1\right)  -\left(
\beta-2\right)  \frac{\rho+\sigma}{2}\right]  u_{1}u_{2}=0\ .
\end{equation}

\end{itemize}

Since for our choice of the values of the parameters of the model,
\begin{align}
\lambda_{2}\left(  0\right)   &  =\frac{\sigma^{2}+1-\sqrt{\left(  \sigma
^{2}-1\right)  ^{2}+\left[  \frac{\left(  \rho+\sigma\right)  }{2}%
+\sigma\left(  \sigma-1\right)  \right]  ^{2}}}{2}<0\\
\lambda_{2}\left(  -\frac{\left(  \rho+\sigma\right)  }{2}\right)   &
=\frac{\sigma^{2}+1-\sqrt{\left(  \sigma^{2}-1\right)  ^{2}+\left[
\sigma\left(  \sigma-1\right)  -\left(  \beta-2\right)  \frac{\left(
\rho+\sigma\right)  }{2}\right]  ^{2}}}{2}<0
\end{align}
the intersection of $\mathcal{E}^{\prime}$ with the planes $z=q_{3}=0$ and
$z=q_{3}=\frac{\left(  \rho+\sigma\right)  }{2}$ are the straight lines. The
manifold in $R^{3}$ corresponding to the relative maxima of $C\left(
t\right)  $ is
\begin{align}
\Sigma &  :=\left\{  u\in\mathbb{R}^{3}:\dot{C}\left(  u\right)
=0\ ,\ \ddot{C}\left(  u\right)  \leq0\right\} \label{Sigma}\\
&  =\left\{  u\in\mathbb{R}^{3}:\left\{
\begin{array}
[c]{l}%
\sigma u_{1}^{2}+u_{2}^{2}+\beta\left(  u_{3}+\frac{\left(  \rho
+\sigma\right)  }{2}\right)  ^{2}=\frac{\beta\left(  \rho+\sigma\right)  ^{2}%
}{4}\\
\sigma^{2}u_{1}^{2}+u_{2}^{2}-\left[  \sigma\left(  \sigma-1\right)  +\left(
\beta-1\right)  \left(  u_{3}+\frac{\rho+\sigma}{2}\right)  +\frac{\rho
+\sigma}{2}\right]  u_{1}u_{2}+\\
+\beta^{2}\left(  u_{3}+\frac{\rho+\sigma}{2}\right)  \left(  u_{3}%
+\rho+\sigma\right)  \leq0
\end{array}
\right.  \right\}  \ .\nonumber
\end{align}
Since for our choice of the parameters
\begin{equation}
\frac{1}{\beta-1}\left[  \sigma\left(  \sigma-3\right)  +\beta\frac{\left(
\rho+\sigma\right)  }{2}\right]  >\rho+\sigma\ ,
\end{equation}
$\Sigma$ is composed by two closed surfaces in $R^{3},\Sigma_{+}$ and
$\Sigma_{-},$ such that $R\Sigma_{+}=\Sigma_{-}$ and intersecting only in the
critical point $c_{0}.$

By definition, $\forall u\in\mathcal{E},$ the vector $\dot{u}\left(  u\right)
$ is orthogonal to the vector $\nabla C\left(  u\right)  ,$ hence it belongs
to the plane spanned by $\nabla E\left(  u\right)  -\left(  \nabla
E\cdot\nabla C\right)  \left(  u\right)  \nabla C\left(  u\right)  $ and
$\left(  \nabla C\times\nabla E\right)  \left(  u\right)  ,$ where
\begin{equation}
\nabla C\times\nabla E=4\left(
\begin{array}
[c]{c}%
u_{2}\left[  \left(  \beta-1\right)  u_{3}+\beta\frac{\rho+\sigma}{2}\right]
\\
-u_{1}\left[  \left(  \beta-\sigma\right)  u_{3}+\beta\frac{\rho+\sigma}%
{2}\right] \\
-\left(  \sigma-1\right)  u_{1}u_{2}%
\end{array}
\right)
\end{equation}
and, since $C$ is a constant of motion for the Hamiltonian field
$v,$\linebreak$\forall u\in\mathcal{E},\ \left(  w\cdot\nabla C\right)
\left(  u\right)  =0.$

Moreover:

\begin{itemize}
\item $\left\vert \nabla E\right\vert \upharpoonleft_{\mathcal{E}},\left\vert
\nabla C\right\vert \upharpoonleft_{\mathcal{E}}$ and $\left\vert \nabla
C\times\nabla E\right\vert \upharpoonleft_{\mathcal{E}}$ are always different
from zero;

\item from (\ref{C'-E}) and (\ref{Sigma})%
\begin{equation}
\ddot{C}\left(  t\right)  =\left(  \dot{u}\cdot\nabla\dot{C}\right)  \left(
t\right)  =\left(  \dot{u}\cdot\nabla\left[  -2\left(  E-\beta\frac{\left(
\rho+\sigma\right)  ^{2}}{4}\right)  \right]  \right)  \left(  t\right)  \ ,
\end{equation}
hence, $\forall u\in\partial\Sigma,\ \dot{u}\left(  u\right)  $ is parallel to
$\nabla C\times\nabla E,$ that is tangent to $\Sigma.$
\end{itemize}

Therefore, $\dot{u}$ is transverse to $\Sigma\backslash\partial\Sigma$ and
since $\forall u\in\Sigma\backslash\partial\Sigma,$%
\begin{equation}
\ddot{C}\left(  u\right)  =\left(  \dot{u}\cdot\nabla\dot{C}\right)  \left(
u\right)  =-2\left(  \dot{u}\cdot\nabla E\right)  <0\;\Longrightarrow\;\left(
\dot{u}\cdot\nabla E\right)  >0\ ,
\end{equation}
the direction of $\dot{u}\left(  u\right)  $ points outward the bounded subset
of $\mathbb{R}^{3},$%
\begin{equation}
\left\{  u\in\mathbb{R}^{3}:E\left(  u\right)  \leq\beta\frac{\left(
\rho+\sigma\right)  ^{2}}{4}\right\}  \ .
\end{equation}

\subsubsection{Parametrization of $\Sigma$}

If $r\in\left(  0,\rho+\sigma\right)  ,\ \gamma:=rB\cap\mathcal{E}$ is a
regular closed curve. Therefore, we can parametrize $\Sigma_{+}$ choosing an
appropriate arc of $\gamma$ as coordinate curve of the parametrization, that
is there exist an open regular subset $\Omega$ of $\mathbb{R}^{2}$ and a
map\linebreak$b^{+}\in C^{1}\left(  \Omega,\mathbb{R}^{3}\right)  \cap
C\left(  \overline{\Omega},\mathbb{R}^{3}\right)  $ such that%
\begin{equation}
\left\{
\begin{array}
[c]{c}%
u_{1}=b_{1}^{+}\left(  y,z\right) \\
u_{2}=b_{2}^{+}\left(  y,z\right) \\
u_{3}=b_{3}^{+}\left(  y,z\right)
\end{array}
\right.  \ ,\ \left(  y,z\right)  \in\Omega\ . \label{b+}%
\end{equation}
Moreover, we can choose the parametrization such that $z=r^{2},$ therefore the
coordinate curves $b_{z}^{+}\left(  y\right)  $ satisfy the equations
\begin{equation}
\left\{
\begin{array}
[c]{c}%
C\left(  b_{z}^{+}\left(  y\right)  \right)  =z\\
\dot{C}\left(  b_{z}^{+}\left(  y\right)  \right)  =0
\end{array}
\right.  \ .
\end{equation}
We remark that the tangent field to $\gamma$ is parallel to $\nabla
C\times\nabla E,$ while, if $\zeta$ denotes the coordinate curve $b_{y}%
^{+}\left(  z\right)  ,$ the tangent field to $\zeta$ is parallel to $\nabla
E\times\left(  \nabla C\times\nabla E\right)  .$

Similar arguments also hold for $\Sigma_{-}.$ Furthermore,
\begin{equation}
\overline{\Omega}\ni\left(  y,z\right)  \longmapsto b^{-}\left(  y,z\right)
:=Rb^{+}\left(  y,z\right)  \in\mathbb{R}^{3}%
\end{equation}
is easily seen to be a parametrization of $\Sigma_{-}$ sharing the same
properties of $b^{+}.$

\subsubsection{Return maps on $\Sigma$}

The evolution of the system maps in itself the ball $\beta\left(  \rho
+\sigma\right)  B,\ \mathcal{E}\subset\beta\left(  \rho+\sigma\right)  B$ and the
velocity field is transverse to $\Sigma\backslash\partial\Sigma$ and
$\mathcal{E}\backslash\Sigma.$

For our choice of parameters $\rho,\beta$ and $\sigma,$ it has been shown in
\cite{T} that there exist periodic orbits crossing a two-dimensional compact
domain $\Delta$ contained in the plane $\pi:=\left\{  u\in\mathbb{R}^{3}%
:u_{3}=1-\left(  \rho+\sigma\right)  \right\}  ,$ which is also intersected by
the stable manifold of the system $W_{o}^{s}$ along some curve $\Gamma_{0}.$
Furthermore, the first eight shortest periodic orbits have been rigorously
found in \cite{GT}. Notice that by symmetry if $\varphi\left(  t,u\right)
,\ u\in\Sigma_{+},$ is a periodic orbit, $R\varphi\left(  t,u\right)  $ is
also a periodic orbit and either $R\varphi\left(  t,u\right)  =\varphi\left(
t,u\right)  $ or $R\varphi\left(  t,u\right)  =\varphi\left(  t,Ru\right)  ,$
where $Ru\in\Sigma_{-};$ that is periodic orbits are either symmetric or
appear in couples whose elements are mapped one into another by $R,$ as
already remarked in \cite{Sp}.

Since $\Delta$ is easily seen to be contained in
\begin{equation}
\left\{  u\in\mathbb{R}^{3}:\dot{C}\left(  u\right)  \leq0\right\}  \cap\pi\ ,
\end{equation}
these periodic orbits then, necessarily cross $\Sigma$ which is also possibly
intersected by $W_{o}^{s}$ along some curve $\Gamma$ lying in the half-space
\begin{equation}
\left\{  u\in\mathbb{R}^{3}:u_{3}\geq1-\left(  \rho+\sigma\right)  \right\}
\ .
\end{equation}

Therefore, if $u_{0}\in\Sigma\backslash\Gamma$ lies on a periodic orbit of
period $t_{0},$ there exists an open neighborhood $N\ni u_{0}$ and a
$C^{1}\left(  N,\mathbb{R}\right)  $ map $\tau$ such that $\tau\left(
u_{0}\right)  =t_{0}$ and $\varphi_{\tau\left(  u\right)  }\left(  u\right)
\in\Sigma$ for any $u\in N.$ Then,
\begin{equation}
N\cap\Sigma\backslash\Gamma\ni u\longmapsto P_{\Sigma}\left(  u\right)
:=\varphi_{\tau\left(  u\right)  }\left(  u\right)  \in\Sigma\ .
\end{equation}
Moreover, it has been proved in \cite{T} that $\Delta\backslash\Gamma_{0}$ is
forward invariant under the return map on $\pi$ and that on $\Delta$ there
estists a forward invariant unstable cone field. These properties are also
shared by a compact subset $\Delta^{\prime}\subset\Sigma$ such that any open
subset of $\Delta^{\prime}$ is diffeomorphic to a open subset of $\Delta.$
Hence, $P_{\Sigma}$ admits an invariant stable foliation with $C^{1+\iota
},\ \iota\in\left(  0,1\right)  $ leaves.

\subsubsection{Construction of the map $T$}

Let $P^{\left(  \pm\right)  }:=P_{\Sigma_{\pm}}.$ By the parametrization
previously introduced for $\Sigma_{+},$ there exists an open subset
$\Omega^{\prime}\subset\Omega\backslash\Gamma^{\prime},$ with $\Gamma^{\prime
}:=\left(  b^{+}\right)  ^{-1}\left(  \Gamma\right)  ,$ and a $C^{1}\left(
\Omega^{\prime},\mathbb{R}^{2}\right)  $ map
\begin{equation}
\Omega^{\prime}\ni\left(  y,z\right)  \longmapsto S\left(  y,z\right)
\in\Omega^{\prime}\ .
\end{equation}
such that, $\forall\left(  y,z\right)  \in\Omega^{\prime},$%
\begin{equation}
\left(  b^{+}\circ S\right)  \left(  y,z\right)  :=\left(  P^{\left(
+\right)  }\circ b^{+}\right)  \left(  y,z\right)  \ ,
\end{equation}
Furthermore,
\begin{equation}
G\left(  y,z\right)  :=\left(  \dot{C}\circ b^{+}\circ S\right)  \left(
y,z\right)  =0\ .
\end{equation}
Let $S_{1},S_{2}$ be respectively the first and the second component of $S.$
Since $b^{+}$ is a diffeomorphism and the components of $\nabla E=\nabla
\dot{C}$ are different from zero on $\Sigma_{+},\ \forall\left(  y,z\right)
\in\Omega^{\prime},\ \partial_{y}G\left(  y,z\right)  \neq0.$ Thus, by the
implicit function theorem, $\forall\left(  y_{0},z_{0}\right)  \in
\Omega^{\prime},$ there exist two open interval $\left(  y_{1},y_{2}\right)
,\ \left(  z_{1},z_{2}\right)  $ such that $\left(  y_{0},z_{0}\right)
\in\left(  y_{1},y_{2}\right)  \times\left(  z_{1},z_{2}\right)
\subseteq\Omega^{\prime}$ and a unique $C^{1}\left(  \left(  z_{1}%
,z_{2}\right)  \right)  $ map
\begin{equation}
\left(  z_{1},z_{2}\right)  \ni z\longmapsto y:=U\left(  z\right)  \in\left(
y_{1},y_{2}\right)
\end{equation}
such that, $\forall z\in\left(  z_{1},z_{2}\right)  ,\ G\left(  U\left(
z\right)  ,z\right)  =0$ and, $\forall\left(  y,z\right)  \in\left(
y_{1},y_{2}\right)  \times\left(  z_{1},z_{2}\right)  $ such that $y\neq
U\left(  z\right)  ,\ G\left(  y,z\right)  \neq0.$

Therefore, let
\begin{equation}
\left(  z_{1},z_{2}\right)  \ni z\longmapsto V\left(  z\right)  :=S_{2}\left(
U\left(  z\right)  ,z\right)  \in\left(  z_{1},z_{2}\right)  \ .
\end{equation}

Notice that, since $b^{+}\in C^{1}\left(  \Omega\right)  ,\ S=\left(
b^{+}\right)  ^{-1}\circ P^{\left(  +\right)  }\circ b^{+}$ is $C^{1}\left(
\Omega^{\prime}\right)  $ if and only if $P^{\left(  +\right)  }$ is, and so
are $U$ and $V.$

Moreover, by symmetry,
\begin{equation}
b^{-}\circ S=Rb^{+}\circ S=RP^{\left(  +\right)  }\circ b^{+}=RP^{\left(
+\right)  }\circ Rb^{-}=P^{\left(  -\right)  }\circ b^{-}\ .
\end{equation}
Hence $S=\left(  b^{-}\right)  ^{-1}\circ P^{\left(  -\right)  }\circ b^{-}.$

Clearly, $\left[  z_{1},z_{2}\right]  \subseteq\left[  m\vee\left(  r^{\ast
}\right)  ^{2},\rho+\sigma\right]  ,$ with $r^{\ast}:=\inf\{r>0:rB\cap
\Sigma\neq\varnothing\}.$

Let $u_{+}\in\Sigma_{+},\left(  y_{+},z_{+}\right)  \in\overline{\Omega}$ be
such that $P^{\left(  +\right)  }\left(  u_{+}\right)  =c_{0}$ and\linebreak%
$b^{+}\left(  y_{+},z_{+}\right)  =u_{+}.$ Setting
\begin{equation}
\left[  z_{1},z_{2}\right]  \ni z\longmapsto X\left(  z\right)  :=\frac
{z-z_{1}}{z_{2}-z_{1}}\in\left[  0,1\right]  \ ,
\end{equation}
we define
\begin{equation}
\left[  0,1\right]  \ni s\longmapsto T\left(  s\right)  :=X\circ V\circ
X^{-1}\left(  s\right)  \in\left[  0,1\right]  \ .
\end{equation}

Hence, by construction, $T$ is a $C^{1}\left(  \left(  0,1\right)
\backslash\left\{  x_{0}\right\}  \right)  $ map, where $x_{0}:=X\left(
z_{+}\right)  ,$ and, since there exits $\iota\in\left(  0,1\right)  $ such
that $P_{\Sigma}$ admits an invariant stable foliation with $C^{1+\iota}$
leaves, then $T$ is also $C^{1+\iota}\left(  \left(  0,1\right)
\backslash\left\{  x_{0}\right\}  \right)  .$

\section{The invariant density for the evolution under $T$}

In this section we compute the density of the unique (by ergodicity)
absolutely continuous invariant measure for the map $T$ and we prove its
statistical stability. For the construction of the density we use the
techniques recently introduced in the paper \cite{CHMV} (see also \cite{BH}
for results related to similar maps), which dealt with Lorenz maps admitting
indifferent fixed points besides points with unbounded derivative. For the
statistical stability we will follow the recent article \cite{BV}, but with
some new substantial improvements. The techniques used in \cite{CHMV} turned
around Young's towers \cite{LSY} and, substantiated by a careful analysis of
the distortion, led to a detailed study of the density of the absolutely
continuous invariant measure, of the recurrence properties of the dynamics and
of limit theorems for H\"{o}lder continuous observables. This analysis could
in particular be carried over when the map has a derivative larger than one at
the fixed point, as in the case we are going to treat, but possibly smaller
than one in some other point (see below). We remind that whenever the Lorenz
map is a Markov expanding map with finite derivative, it could be investigated
with the spectral techniques of Keller \cite{K}. Young's towers are useful
when the map looses the Markov property, but preserves points with unbounded
derivative. This has been analysed in \cite{KDO}, see also \cite{OHL} when
there are critical points too. Our main effort will be in investigating the
smoothness of the density. We will show that such a density is Lipschitz
continuous on the whole unit interval but in one point. The argument we
produce is a strong improvement with respect to the result achieved in
\cite{CHMV} (and applies to it as well), where we solely proved the Lipschitz
continuity on countably many intervals partitioning the unit interval. We
stress that, as far as we know,this is the first result where the smoothness
of the density for Lorenz like maps is explicitly exhibited.\newline

\textbf{Notations:} With $a_{n}\approx b_{n}$ we mean that there exists a
constant $C\geq1$ such that $C^{-1}b_{n}\leq a_{n}\leq Cb_{n}$ for all
$n\geq1;$ with $a_{n}\sim b_{n}$ we mean that $\lim_{n\rightarrow\infty}
\frac{a_{n}}{b_{n}}=1.$ We will also use the symbols "$O$" and "$o$" in the
usual sense. \newline

The analysis we perform in this section applies to a large class of
Lorenz-like maps which includes in particular those whose behavior is given by
the theoretical arguments of the preceding section and by the numerical
investigations of the paper \cite{PM}. The map $T$ (fig.1) has a left and a
right convex branches around the point $0<x_{0}<1;$ the left branch is
monotonically increasing and uniformly expanding even at the fixed point $0,$
while the right one is monotonically decreasing with the derivative bounded
from below by a constant less than one; at the cusp, located at $x_{0},$ the
left and right derivative blow up to infinity. Both branches are onto $[0,1]$
and this makes the map Markovian. Moreover, we recall that our map is $C^{1}$
on $[0,1]\backslash\{x_{0}\}$ and $C^{1+\iota},$\linebreak$\iota\in\left(
0,1\right)  ,$ on $(0,x_{0})\cup(x_{0},1).$

\begin{figure}[htbp]
\centering
\resizebox{0.75\textwidth}{!}{%
\includegraphics{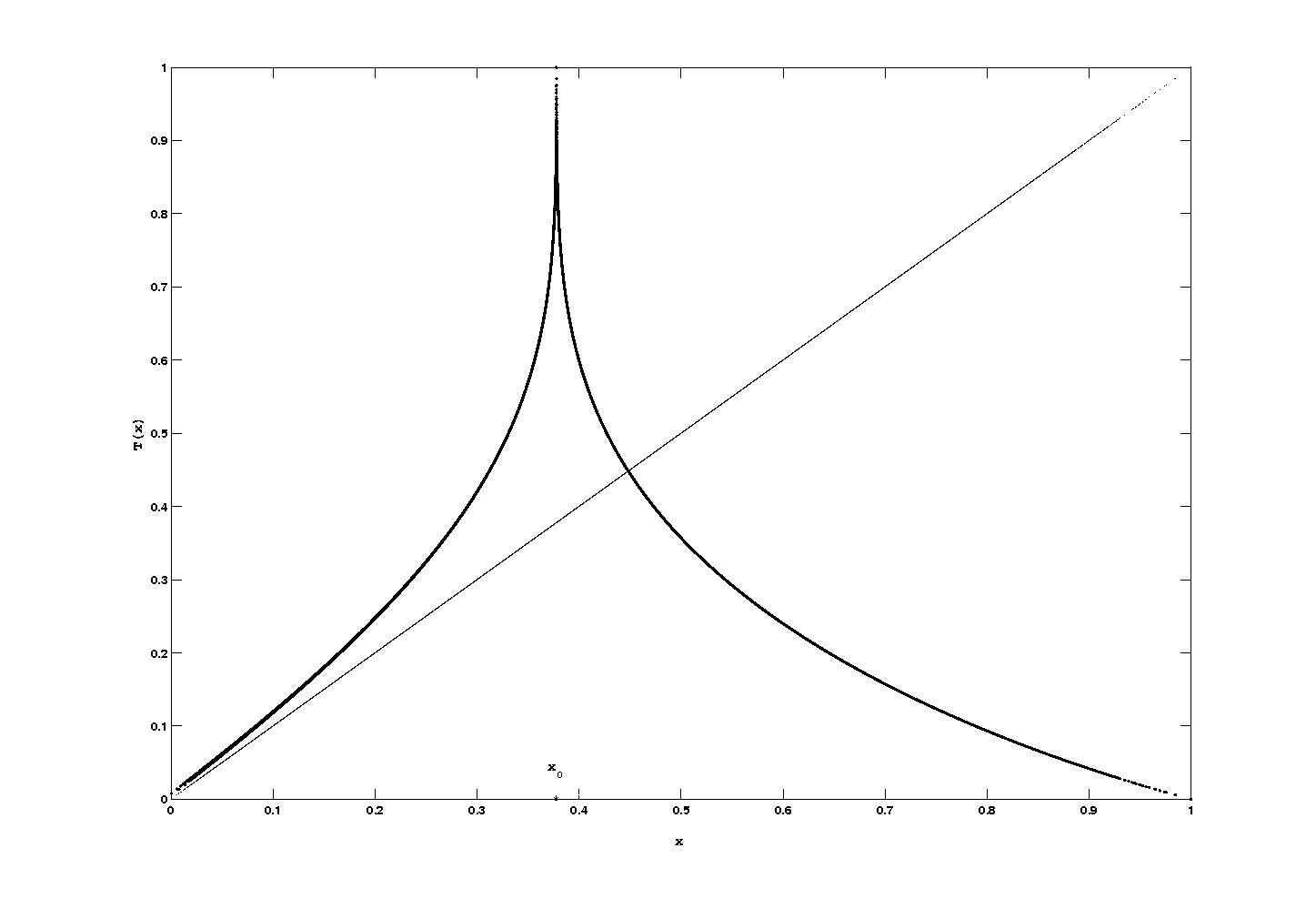}
}
%If not, use
%\vspace{5cm}       % Give the correct figure height in cm
\caption{Normalized Lorenz cusp map for the Casimir maxima}
\label{fig:1}       % Give a unique label
\end{figure}

The local behaviors are ($c$ will denote a positive constant which could take
different values from one formula to another):
\begin{align}
&  \left\{
\begin{array}
[c]{l}%
T(x)=\alpha^{\prime}x+\beta^{\prime}x^{1+\psi}+o(x^{1+\psi});\ x\rightarrow
0^{+}\\
DT(x)=\alpha^{\prime}+cx^{\psi}+o(x^{\psi}),\ \alpha^{\prime}>1;\ \beta
^{\prime}>0;\ \psi>1
\end{array}
\right.  \ ,\label{T_DT1}\\
&  \left\{
\begin{array}
[c]{l}%
T(x)=\alpha(1-x)+\tilde{\beta}(1-x)^{1+\kappa}+o((1-x)^{1+\kappa
});\ x\rightarrow1^{-}\\
DT(x)=-\alpha-c(1-x)^{\kappa}+o((1-x)^{\kappa}),\ 0<\alpha<1,\ \tilde{\beta
}>0,\ \kappa>1
\end{array}
\right.  \ ,\label{T_DT2}\\
&  \left\{
\begin{array}
[c]{l}%
T(x)=1-A^{\prime}(x_{0}-x)^{B^{\prime}}+o((x_{0}-x)^{B^{\prime}}%
);\ x\rightarrow x_{0}^{-},\ A^{\prime}>0\\
DT(x)=c(x_{0}-x)^{B^{\prime}-1}+o((x_{0}-x)^{B^{\prime}-1}),\ 0<B^{\prime}<1
\end{array}
\right.  \ ,\label{T_DT3}\\
&  \left\{
\begin{array}
[c]{l}%
T(x)=1-A(x-x_{0})^{B}+o((x-x_{0})^{B});\ x\rightarrow x_{0}^{+},\ A>0\\
DT(x)=-c(x-x_{0})^{B-1}+o((x-x_{0})^{B-1}),\ 0<B<1
\end{array}
\right.  \ . \label{T_DT4}%
\end{align}

We set $B^{\ast}:=\max(B,B^{\prime});$ moreover we set $T_{1}$ (resp. $T_{2}$)
the restriction of $T$ to $[0,x_{0}]$ (resp. to $[x_{0},1]$). A key role is
played by the preimages of $x_{0}$ since they will give the sets where we will
induce with the first return map; so we set: $a_{0}:=T_{2}^{-1}x_{0};$%
\ $a_{0}^{\prime}:=T_{1}^{-1}x_{0};$\ $a_{p}^{\prime}=T_{1}^{-p}a_{0}^{\prime
};$\ $a_{p}=T_{2}^{-1}T_{1}^{-(p-1)}a_{0}^{\prime},$\ $p\geq1.$ We also define
the sequences $\{b_{p}\}_{p\geq1}\subset(x_{0},a_{0})$ and $\{b_{p}^{\prime
}\}_{p\geq1}\subset(a_{0}^{\prime},x_{0})$ as $Tb_{p}^{\prime}=Tb_{p}%
=a_{p-1}.$ The idea is now to induce on some domain $I$ and to replace the
action of $T$ on $I$ with that of the first return map $T_{I}$ into $I.$ We
will see that the systems $(I,T_{I})$ will admit an absolutely continuous
invariant measure $\mu_{I}$ which is in particular equivalent to the Lebesgue
measure with a density $\rho_{I}$ bounded from below and from above. There
will be finally a link between the induced measure $\mu_{I}$ and the
absolutely continuous invariant measure $\mu$ on the interval, which will
allows us to get some informations on the density $\rho$ of $\mu.$ The
principal set where we will choose to induce is the open interval
$I=(a_{0}^{\prime},a_{0})\backslash\{x_{0}\}.$ The subsets $Z_{p}\subset I$
with first return time $p$ will have the form
\begin{align}
Z_{1}  &  =(a_{0}^{\prime},b_{1}^{\prime})\cup(b_{1},a_{0})\label{SI}\\
Z_{p}  &  =(b_{p-1}^{\prime},b_{p}^{\prime})\cup(b_{p},b_{p-1})\quad p>1\ .
\label{SII}%
\end{align}
We will also induce over the open sets $(a_{n}^{\prime},a_{n-1}^{\prime})$ and
$(a_{n},a_{n+1}),n>1,$ simply denoted in the following as the rectangles
$I_{n}.$ In order to apply the techniques of \cite{CHMV}, we have to show that
the induced maps are aperiodic uniformly expanding Markov maps with bounded
distortion on each set with prescribed return time. On the sets $I_{n}$ the
first return map $T_{I_{n}}$ is Bernoulli, while the aperiodicity condition on
$I$ follows easily by the inspection of the graph of the first return map
$T_{I}:I\rightarrow I$ showing that it maps: $(a_{0}^{\prime},b_{1}^{\prime})$
onto $(x_{0},a_{0});$ the intervals $(b_{l}^{\prime},b_{l+1}^{\prime}%
),\ l\geq1,$ onto the interval $(a_{0}^{\prime},x_{0})$ and $(b_{1},a_{0})$
onto $(x_{0},a_{0}).$ Finally, $T_{I}$ sends the intervals $(b_{l+1}%
,b_{l}),\ l\geq1$ onto $(a_{0}^{\prime},x_{0}).$ Bounds on the distortion of
the first return map on $I$ and on the $I_{n}$ can be proved exactly in the
same way as in Proposition 3 of \cite{CHMV} (we defer to it for the details)
provided we show that the first return maps are uniformly
expanding\footnote{The Lorenz-like map considered in \cite{CHMV} was $C^{2}$
outside the boundary points and the cusp; our map is instead $C^{1+\iota}.$
This will not change the proof of the distortion in \cite{CHMV} and all the
statistical properties which follow from it. As a matter of fact, in the
initial formula (5) in \cite{CHMV}, we have to replace the term $\left\vert
\frac{D^{2}T(\xi)}{DT(\xi)}\right\vert \left\vert T^{q}(x)-T^{q}(y)\right\vert
$ with $\frac{C_{h}}{\left\vert DT(\xi)\right\vert }\left\vert T^{q}%
(x)-T^{q}(y)\right\vert ^{\iota},$ where $C_{h}$ is the H\"{o}lder constant
larger than $0$ depending only on $T$ and $\xi$ is a point between the
iterates $T^{q}(x),\ T^{q}(y).$ The only delicate point where the $C^{1+\iota
}$ assumption could give problems is the summability of the series at point
(i) in the statement of Lemma 4 in \cite{CHMV}. The general term of this
series will be of the form (we adapt to our case): $(a_{n+1}-a_{n})^{\iota}.$
In \cite{CHMV}, due to the presence of the indifferent fixed point, the term
$(a_{n+1}-a_{n})$ decays polynomially like $n^{-\kappa},$ say, where
$\kappa>1$ depends on the map. In order to guarantee the aforecited
summability property we have therefore to ask an additional assumption on
$\kappa,$ namely $\kappa>\iota^{-1}$. We do not have such a constraint in our
case since the length $(a_{n+1}-a_{n})$ decays exponentially fast.}. The proof
of this fact is given in the next Lemma and it requires a few assumptions
which can be checked numerically with a finite number of steps and by a direct
inspection of the graph of $T.$ With abuse of language we will say that the
derivative is larger than $1$ if its absolute value is larger than $1.$

\begin{lemma}
Let us suppose that in addition to (\ref{T_DT1})-(\ref{T_DT4}) the map $T$
satisfies the assumptions:

\begin{itemize}
\item[(i)] $d_{(1,0)}:=\inf_{x\in(b_{1},a_{0})}|DT(x)|>1;$

\item[(ii)] $|DT(b_{1})|\geq DT(a_{0}^{\prime});$

\item[(iii)] $|DT(a_{p-1})|DT(a_{p-2}^{\prime})\cdots DT(a_{1}^{\prime
})DT(a_{0}^{\prime})>\alpha^{\prime\prime},\ \forall p\leq p^{\ast
}:=\left\lfloor 1+\frac{\log(\alpha^{\prime\prime}\alpha^{-1})}{\log
\alpha^{\prime}}\right\rfloor ,$ for $1<\alpha^{\prime\prime}\leq
d_{(1,0)}\wedge\alpha^{\prime}.$
\end{itemize}

Then the first return time maps $T_{I}$ and $T_{I_{n}},n>1$ have the
derivative uniformly bounded below away from $\alpha^{\prime\prime}.$
\end{lemma}

\begin{proof}
We give the proof for $I$ and we generalize after to all the $I_{n}.$ We
represent with an arrow $"\rightarrow"$ the evolution under $T$ of a subset
$Z_{p}\subset I,\ p\geq1,$ given in (\ref{SI}) and (\ref{SII}). Consequently
$(b_{1},a_{0})\rightarrow(x_{0},a_{0})$ and $(a_{0}^{\prime},b_{1}^{\prime
})\rightarrow(x_{0},a_{0}).$ In the latter case the derivative of the map
coincides with that of $T$ and is larger than $1,$ since $T_{1}$ has
derivative larger than $\alpha^{\prime}>1.$ The former case follows by
condition (i). For $p>1$ we have:
\begin{equation}
\left\{
\begin{array}
[c]{l}%
\left(  b_{2},b_{1}\right)  \rightarrow\left(  a_{0},a_{1}\right)
\rightarrow\left(  a_{0}^{\prime},x_{0}\right)  , \ p=2\\
(b_{p},b_{p-1})\rightarrow(a_{p-2},a_{p-1})\rightarrow(a_{p-2}^{\prime
},a_{p-3}^{\prime})\rightarrow(a_{p-3}^{\prime},a_{p-4}^{\prime}%
)\rightarrow\cdots\\
\rightarrow(a_{1}^{\prime},a_{0}^{\prime})\rightarrow(a_{0}^{\prime}%
,x_{0})\quad p\geq3
\end{array}
\right.
\end{equation}
and
\begin{equation}
\left\{
\begin{array}
[c]{l}%
\left(  b_{2}^{\prime},b_{1}^{\prime}\right)  \rightarrow\left(  a_{0}%
,a_{1}\right)  \rightarrow\left(  a_{0}^{\prime},x_{0}\right)  , \ p=2\\
(b_{p-1}^{\prime},b_{p}^{\prime})\rightarrow(a_{p-2},a_{p-1})\rightarrow
(a_{p-2}^{\prime},a_{p-3}^{\prime})\rightarrow(a_{p-3}^{\prime},a_{p-4}%
^{\prime})\rightarrow\cdots\\
\rightarrow(a_{1}^{\prime},a_{0}^{\prime})\rightarrow(a_{0}^{\prime}%
,x_{0})\quad p\geq3
\end{array}
\right.
\end{equation}
In order to get that the derivative of $T^{p}$ is larger than one we need:

\begin{itemize}
\item in the first case
\begin{equation}
\left\vert DT(b_{p-1})DT(a_{p-1})\right\vert DT(a_{p-2}^{\prime})\cdots
DT(a_{2}^{\prime})DT(a_{1}^{\prime})>1\ ; \label{(*)}%
\end{equation}

\item in the second case
\begin{equation}
DT(b_{p-1}^{\prime})\left\vert DT(a_{p-1})\right\vert DT(a_{p-2}^{\prime
})\cdots DT(a_{2}^{\prime})DT(a_{1}^{\prime})>1\ . \label{(**)}%
\end{equation}

\end{itemize}

Let us suppose now that we have, for $p>1,$
\begin{equation}
\left\vert DT(a_{p-1})\right\vert DT(a_{p-2}^{\prime})\cdots DT(a_{1}^{\prime
})DT(a_{0}^{\prime})>\alpha^{\prime\prime}\ . \label{C1}%
\end{equation}
If this condition holds, then the inequality (\ref{(**)}) follows too with the
same uniform (in $p$) bound $\alpha^{\prime\prime},$ since, by monotonicity of
the first derivative, $DT(b_{p-1}^{\prime})>DT(a_{0}^{\prime}).$ In order to
satisfy the inequality (\ref{(*)}) with a lower bound given by $\alpha
^{\prime\prime},$ by assuming again (\ref{C1}), it will be sufficient to show
that $\left\vert DT(b_{p-1})\right\vert \geq DT(a_{0}^{\prime}),$ which, by
monotonicity, is implied by $|DT(b_{1})|\ge DT(a^{\prime}_{0})$ and this
follows from assumption (ii). Therefore, we are left with the proof of the
validity of condition (\ref{C1}). By condition (iii) this holds true for
$p\leq p^{\ast}.$ Moreover, since the points $a_{p-2}^{\prime},a_{p-3}%
^{\prime},\dots,a_{0}^{\prime}$ lies on the left of $x_{0},$ all the $(p-1)$
derivatives in the block $DT(a_{p-2}^{\prime})\cdots DT(a_{2}^{\prime
})DT(a_{0}^{\prime}))$ are larger $\alpha^{\prime}.$ On the other hand, the
derivative in $a_{p-1}$ is surely larger than $\alpha.$ Hence, (\ref{C1})
holds for all $p$ such that $\alpha\left(  \alpha^{\prime}\right)
^{p-1}>\alpha^{\prime\prime}.$\newline Now we return to the rectangles in
$I_{n}.$ Let us first call \emph{complete path} the graphs given above
starting respectively from $(a_{p-1},a_{p}),\ p>1,$ and ending in
$(a_{0}^{\prime},x_{0})$ and starting from $(a_{p}^{\prime},a_{p-1}^{\prime
}),\ p>1$ and ending in $(a_{0}^{\prime},x_{0}).$ It is easy to check, by
looking at the grammar\footnote{Let us give the coding for the map $T$ with
the grammar which we invoked above. To use a coherent notation we will
redefine $a_{-0}\equiv a_{0}^{\prime};$ $a_{-p}\equiv a_{p}^{\prime}%
=T_{1}^{-p}a_{0}^{\prime}\ ,\ p\geq1.$ We associate with each point
$x\in\lbrack0,1]\backslash\chi,$ where $\chi=\cup_{i\geq0}T^{-i}\{x_{0}\},$
the unique coding $x=(\omega_{0},\omega_{1},\dots,\omega_{n},\dots
),\ \omega_{l}\in\mathbb{Z},$ where (from now on $n$ will denote a positive
integer larger than $1$), $\omega_{l}=n$ iff $T^{l}x\in(a_{n-1},a_{n}%
);\ \omega_{l}=-n$ iff $T^{l}x\in(a_{-n},a_{-(n-1)});$ $\omega_{l}=0$ iff
$T^{l}x\in I.$ The grammar is the following (the formal symbol $-0$ must be
intended as $0$) :%
\begin{align*}
\omega_{i}  &  =n>0\Rightarrow\omega_{i+1}=-n\ ;\ \omega_{i}=-n\Rightarrow
\omega_{i+1}=-(n-1)\\
\newline\omega_{i}  &  =0\Rightarrow\omega_{i+1}=n\geq0\ (\text{any }n)
\end{align*}
} given by the arrows, that any subset of the rectangles in $I_{n}$ with first
return time $q>n$ will contain points whose trajectory follows a complete
path, or spends some time in $(x_{0},a_{0}).$ In any case, and by the
condition (\ref{C1}) whose validity has been checked above, the derivative
$DT^{q}$ will be strictly larger than $\alpha^{\prime\prime}.$
\end{proof}

\begin{remark}
The assumption (i)-(iii) in the previous Lemma are easily verified for the map
investigated in \cite{PM}. In particular, with the values $\alpha=0.4603$ and
$\alpha^{\prime}=1.113$ associated to the map and with $\alpha^{\prime\prime
}\sim1.01,$ the inequality (iii) is verified for $p\geq9;$ hence we have only
to check (\ref{C1}) for $1<p\leq8$ and this has been done, and confirmed, by a
direct easy numerical inspection.
\end{remark}

As a consequence of the preceding results, we could apply, as in \cite{CHMV},
the L.-S. Young tower theory and conclude the following statements:

\begin{itemize}
\item On the induced set $I,$ the tail of the Lebesgue measure of the set of
points with first return bigger than $n,$ to be more precise the quantity
$\sum_{k>n}m\{x\in I~;~\tau_{I}(x)\geq k\},$ where $\tau_{I}(x)$ denotes the
first return of the point $x$ into $I,$ decays exponentially fast with $n.$ By
using (\ref{SI}) and the asymptotic values for the $b_{n}$ and $b_{n}^{\prime
}$ given below it is immediate to find that the previous rate of decay is
$O(\left(  \alpha^{\prime}\right)  ^{-\frac{n}{B^{\ast}}}).$ This implies the
existence on the Borel $\sigma$-algebra $\mathcal{B}\left(  [0,1]\right)  $ of
an absolutely continuous invariant measure $\mu$ with exponential decay of
correlations for H\"{o}lder observables evolving under $T$ w.r.t. $\mu$ (the
rate of this decay will be of the type $\hat{\alpha}^{-n},$ where $\hat
{\alpha}$ is possibly different from $\alpha^{\prime}$).

\item Since the first return maps $T_{I}$ and $T_{I_{n}}$ are aperiodic
uniformly expanding Markov maps, they admit invariant measures $\mu_{I}$ and
$\mu_{I_{n}}$ which turn out to be equivalent to Lebesgue on $I$ and $I_{n}$
with densities bounded away from $0$ and $\infty$ \cite{CHMV} and also
Lipschitz continuous on the images of the rectangles of the their associated
Markov partition \cite{AD}\footnote{It is argued in \cite{AD} that if $\alpha$
is a Markov partition of the standard probability metric space $(X,\mathcal{B}%
,m,T)$ with distance $d,$ then $T\alpha\subset\sigma(\alpha),$ where
$\sigma(\alpha)$ denotes the sigma-algebra generated by the partition
$\alpha,$ and therefore it exists a (possibly countable) partition $\beta$
coarser than $\alpha$ such that $\sigma(T\alpha)=\sigma(\beta).$ Moreover, if
the system is Gibbs-Markov, as in our case, then the space $Lip_{\infty,\beta
}$ of functions $f:X\rightarrow\mathbb{R},\ $ $f\in L_{m}^{\infty}%
:=L_{m}^{\infty}\left(  X\right)  ,$ which are Lipschitz continuous on each
$Z\in\beta,$ is a Banach space with the norm: $\left\Vert f\right\Vert_{Lip_{\infty,\beta}%
}=\left\Vert f\right\Vert_{L_{m}^{\infty}}+D_{\beta}f,$ where $D_{\beta}f=\sup_{Z\in\beta}%
\sup_{x,y\in Z}\frac{|f(x)-f(y)|}{d(x,y)}.$ The space $Lip_{\infty,\beta}$ is
compactly injected into $L_{m}^{1},$ which gives the desired conclusions on
the smoothness of the density as a consequence of the Lasota-Yorke inequality.
Notice that in our case $m$ in just the Lebesgue measure. We denote by~$B(I)$
the Banach space $Lip_{\infty,\beta}$ defined on $I.$}. In the sequel we will
show that such densities coincide, but a constant, with the restriction, on
the inducing sets, of the density $\rho$ of the invariant measure $\mu$ for
the map $T.$ Now, the images of the rectangles of the Markov partitions are
the (disjoint) sets $(a_{0}^{\prime},x_{0})$ and $(x_{0},a_{0}),$ when we
induce over $I,$ and the whole intervals $(a_{n}^{\prime},a_{n-1}^{\prime})$
and $(a_{n},a_{n+1}),\ n>1,$ when we induce over the rectangles in $I_{n}.$
Therefore we could conclude that the density of the invariant measure $\mu$ is
a piecewise Lipschitz continuous functions with possible discontinuities at
the points $a_{p},a_{p}^{\prime},\ p>1,$ $a_{0},a_{0}^{\prime}$ and $x_{0}.$
\end{itemize}

We now improve this last result by showing that the density is Lipschitz
continuous over the unit interval but in the cusp point $x_{0}.$ We stress
that this result will improve as well Proposition 13 in \cite{CHMV}.

\begin{proposition}
The density $\rho$ of the invariant measure $\mu$ is Lipschitz continuous and
bounded over the intervals $[0,1]$. Moreover,%
\begin{equation}
\lim_{x\rightarrow0^{+}}\rho(x)=\lim_{x\rightarrow1^{-}}\rho(x)=0\ .
\end{equation}

\end{proposition}

\begin{proof}
We work on the induced set $I.$ The invariant measure $\mu_{I}$ for the
induced map $T_{I}$ is related to the invariant measure $\mu$ over the whole
interval thanks to the well-known formula due to Pianigiani:
\begin{equation}
\mu(B)=C_{r}\sum_{i}\sum_{j=0}^{\tau_{i}-1}\mu_{I}(T^{-j}(B)\cap Z_{i})
\label{reldens}%
\end{equation}
where $B$ is any Borel set in $[0,1]$ and the first sum runs over the
cylinders $Z_{i}$ with prescribed first return time $\tau_{i}$ and whose union
gives $I.$ The normalizing constant $C_{r}=\mu(I)$ satisfies $1=C_{r}\sum
_{i}\tau_{i}\mu_{I}(Z_{i}).$

This immediately implies that by calling $\hat{\rho}$ the density of $\mu_{I}
$ we have that $\rho(x)=C_{r}\hat{\rho}(x)$ for $m$-almost every $x\in I$ and
therefore $\rho$ can be extended to a Lipschitz continuous function on $I $ as
$\hat{\rho}.$ A straightforward application of formula (\ref{reldens}) gives
\cite{CHMV}:
\begin{align}
\mu(a_{n-1},a_{n})  &  =C_{r}\mu_{I}(Z_{n+1})\\
\mu(a_{n}^{\prime},a_{n-1}^{\prime})  &  =C_{r}\sum_{p=n+2}^{\infty}\mu
_{I}(Z_{p})
\end{align}
Let us now take a measurable $B\subset(a_{n}^{\prime},a_{n_{1}}^{\prime});$
the formula above immediately implies that
\begin{equation}
\mu(B)=C_{r}\sum_{p=n+2}^{\infty}\mu_{I}(T^{-(p-n)}B\cap Z_{p})
\end{equation}
Passing to the densities we have
\begin{equation}
\int_{B}\rho\left(  x\right)  dx=C_{r}\sum_{p=n+2}^{\infty}\int_{T^{-(p-n)}%
B\cap Z_{p}}\hat{\rho}\left(  x\right)  dx
\end{equation}
We now perform a change of variables by observing that the set $B$ is pushed
backward $p-n-2$ times by means of $T_{1}^{-1},$ then once by means of
$T_{2}^{-1}$ and finally it splits into two parts according to the actions of
$T_{1}^{-1}$ and $T_{2}^{-1}.$ Therefore,
\begin{equation}
\sum_{p=n+2}^{\infty}\int_{T^{-(p-n)}B\cap Z_{p}}\hat{\rho}\left(  x\right)
dx=\sum_{p=n+2}^{\infty}\sum_{l=1,2}\int_{B}\frac{\hat{\rho}(T_{l}^{-1}%
T_{2}^{-1}T_{1}^{-(p-n-2)}y)}{|DT^{p-n}(T_{l}^{-1}T_{2}^{-1}T_{1}%
^{-(p-n-2)}y)|}dy\ .
\end{equation}
Since $B$ is any measurable set in $(a_{n}^{\prime},a_{n-1}^{\prime}),$ we
have for $m$-almost every $x\in(a_{n}^{\prime},a_{n-1}^{\prime}),$%
\begin{align}
\rho(x)  &  =C_{r}\sum_{p=n+2}^{\infty}\sum_{l=1,2}\frac{\hat{\rho}(T_{l}%
^{-1}T_{2}^{-1}T_{1}^{-(p-n-2)}x)}{|DT^{p-n}(T_{l}^{-1}T_{2}^{-1}%
T_{1}^{-(p-n-2)}x)|}\nonumber\\
&  =C_{r}\sum_{m=2}^{\infty}\sum_{l=1,2}\frac{\hat{\rho}(T_{l}^{-1}T_{2}%
^{-1}T_{1}^{-(m-2)}x)}{|DT^{m}(T_{l}^{-1}T_{2}^{-1}T_{1}^{-(m-2)}x)|}\ .
\label{FF}%
\end{align}
This formula does not depend on the choice of the interval $(a_{n}^{\prime
},a_{n-1}^{\prime})$ and therefore it holds for $x\in(0,a_{0}^{\prime}).$ For
the cylinders $(a_{n-1},a_{n})$ we get similarly that, for $m$-almost any
$x\in(a_{0},1),$%
\begin{equation}
\rho(x)=C_{r}\sum_{l=1,2}\frac{\hat{\rho}(T_{l}^{-1}x)}{|DT(T_{l}^{-1}x)|}\ .
\label{FF2}%
\end{equation}
Since $\hat{\rho}$ is Lipschitz continuous inside $I$ and the inverse branches
of $T$ are $C^{1+\iota},$ we conclude that $\rho$ can be chosen as Lipschitz
continuous over the disjoint open intervals $(0,a_{0}^{\prime})\cup(a_{0},1).$
It is now useful to observe that the right hand sides of (\ref{FF}) and
(\ref{FF2}) give exactly the expression of the Perron-Frobenius operator
associated to the first return map and whenever $x$ is chosen into $I.$ By the
existence of the left (resp. right) limit of $\hat{\rho}$ in $a_{0}^{\prime}$
(resp. $a_{0}$) we immediatly obtain the continuity of $\rho$ in such points.
We use now this result to prove the
continuity of the density in $x_0$. We remind that such a  density
is the fixed point of the Perron-Frobenius operator, so that it
verifies the following equation, for any $x\in[0,1]$ :
\begin{equation}\label{PPFF}
\rho(x)=\frac{\rho(T_{1}^{-1}(x))}{|DT\left(  T_{1}^{-1}(x)\right)  |}%
+\frac{\rho(T_{2}^{-1}(x))}{|DT\left(  T_{2}^{-1}(x)\right)  |}\ ,
\end{equation}
which gives, for $x=x_0$
\begin{equation}
\rho(x_0)=\frac{\rho(a'_0)}{|DT\left(  a'_0\right)  |}%
+\frac{\rho(a_0)}{|DT\left(  a_0\right)  |}\ ,
\end{equation}
and this proves immediately the continuity in $x_0$.
\newline We now observe that assumptions (\ref{T_DT1})-(\ref{T_DT4}) together
with the facts that $T_{2}(a_{p})=a_{p-1}^{\prime},$ $T_{1}(a_{p}^{\prime
})=a_{p-1}^{\prime}$ and $T_{1}b_{p}=T_{2}b_{p}=a_{p-1},$ allow to get easily
the following asymptotic behaviors (for $p$ large) for the preimages of
$x_{0}$ (again $c$ will denote a constant independent of $p$ and that could
change from a formula to another):
\begin{align}
a_{p}^{\prime}  &  \sim\frac{c}{\left(  \alpha^{\prime}\right)  ^{p}%
};\ (1-a_{p})\sim\frac{c}{\left(  \alpha^{\prime}\right)  ^{p}}\ ,\\
(x_{0}-b_{p}^{\prime})  &  \sim\frac{c}{\left(  \alpha^{\prime}\right)
^{\frac{p}{B^{\prime}}}};\ (b_{p}-x_{0})\sim\frac{c}{\left(  \alpha^{\prime
}\right)  ^{\frac{p}{B}}}\ .
\end{align}
These formulas immediately imply that for $x=b_{p}$ (resp. $x=b_{p}^{\prime}$)
in a neighborhood of $x_{0}$ and for $p$ large the derivative behaves like
$|DT(x)|\sim c\left(  \alpha^{\prime}\right)  ^{p(\frac{1}{B}-1)}$ (resp.
$|DT(x)|\sim c(\alpha^{\prime})^{p(\frac{1}{B^{\prime}}-1)}$). Since
$\hat{\rho}$ is bounded away from zero and infinity on $I,$ by the preceding
scalings on the growth of the derivative near $x_{0}$ we have that
$\rho(x)\approx x^{\frac{1}{B^{\ast}}-1}$ for $x$ close to $0$ and $1,$ which
means that $\rho$ can be extended by continuity to zero on the right side of
$0$ and on the left side of $1.$
\end{proof}

The preceding proposition suggests the following scaling for the density.

\begin{proposition}%
\begin{align}
\rho(x)  &  =c^{\prime}x^{a}+o(x^{a}),\ x\rightarrow0^{+};\ a>0,\\
\rho(x)  &  =c^{\prime\prime}(1-x)^{b}+o((1-x)^{b});\ x\rightarrow1^{-},\ b>0,
\end{align}
with
\begin{equation}
a=b=\frac{1}{B^{\ast}}-1
\end{equation}
and the constant $c^{\prime}$ and $c^{\prime\prime}$ verifying
\begin{equation}
\left(  \frac{1}{\alpha^{\prime}}\right)  ^{\frac{1}{B^{\ast}}}+\left(
\frac{1}{\alpha(\frac{c^{\prime}}{c^{\prime\prime}})^{B^{\ast}}}\right)
^{\frac{1}{B^{\ast}}}=1\ .
\end{equation}

\end{proposition}

\begin{proof}
We use again formula (\ref{PPFF}). By using for $T$ and its two inverse branches the asymptotic polynomial
behaviors in $0$ and $1$ given in (\ref{T_DT1})-(\ref{T_DT4}), we get at the
lowest order in $x$ in the neighborhood of $0,$%
\begin{equation}
c^{\prime}\left(  \alpha^{\prime}\right)  ^{-a-1}x^{a}+c^{\prime\prime}%
\alpha^{-b-1}x^{b}=c^{\prime\prime}x^{a}\ .\label{fr}%
\end{equation}
Now, suppose $a<b.$ Then $\left(  \alpha^{\prime}\right)  ^{-a-1}\approx1,$
which implies either $\alpha^{\prime}=1$ or\linebreak$a=-1$ and both cases are
excluded. On the contrary, if $a>b,$ then $\alpha^{-b-1}x^{b}\sim0,$ implying
$\alpha=0,$ which is still impossible. Hence, we necessarily have $a=b.$ We
now take the point $x$ in the neighborhood of $1.$ By explicitating
$T_{1}^{-1}\left(  x\right)  $ and $T_{2}^{-1}\left(  x\right)  $ with respect
to $x$ in the neighborhood of $x_{0}$ and substituting into the
Perron-Frobenius equation we get, at the lowest order in $1-x,$%
\begin{equation}
\frac{O(1)}{\left(  1-x\right)  ^{\frac{B-1}{B}}}+\frac{O(1)}{\left(
1-x\right)  ^{\frac{B^{\prime}-1}{B^{\prime}}}}=(1-x)^{b}\ ,
\end{equation}
from which we obtain
\begin{equation}
(1-x)^{-\frac{B^{\ast}-1}{B^{\ast}}}\approx(1-x)^{b}\ .
\end{equation}
We finally conclude that $a=b=\frac{1}{B^{\ast}}-1.$ Substituting this common
value into equation (\ref{fr}) we finally get the expression relating the
constants $c^{\prime}$ and $c^{\prime\prime}.$
\end{proof}

The latter relation is a good check for the validity of the shape of the
density in $0$ and $1.$

By assuming the continuity of $\rho$ in $x_{0}$, we could use the value of $a
$ given above in terms of the map parameter $B^{\ast}$ to guess a functional
expression for $\rho.$ In agreement with the previous considerations, such an
expression could be
\begin{equation}
\rho(x)=N\left(  \gamma,\delta\right)  e^{-\gamma x}x^{\delta}(1-x)^{\delta
}\ , \label{rfit}%
\end{equation}
where, if $I_{\nu}\left(  z\right)  $ is the modified Bessel function of the
first kind,%
\begin{equation}
N\left(  \gamma,\delta\right)  =\frac{\gamma^{\frac{1}{2}+\delta}%
e^{\frac{\gamma}{2}}}{\sqrt{\pi}\Gamma\left(  1+\delta\right)  I_{\frac{1}%
{2}+\delta}\left(  \frac{\gamma}{2}\right)  }\ ,
\end{equation}
with $\delta=a,\ c^{\prime}=N\left(  \gamma,\delta\right)  $ and
$c^{\prime\prime}=N\left(  \gamma,\delta\right)  e^{-\gamma}.$

Numerical computations performed on about $10^{5}$ values for Casimir maxima
allowed us to estimate the parameters describing the local behavior of the map
listed at the beginning of this section:

\begin{itemize}
\item $\alpha^{\prime}\simeq1.113\ ,\ \alpha\simeq0.4603\ ;$

\item $B^{\prime}\simeq0.3095\ ,\ B\simeq0.2856\ .$
\end{itemize}

Therefore, we get $B^{\ast}=B^{\prime}$ and $\delta\simeq2.2258.$ The fit of
the empirical stationary distribution function performed with such parameters
comes out to be in good agreement with the functional expression for the
invariant density (\ref{rfit}) and the estimated value for $\gamma$ is
$\gamma\simeq4.26.$

\begin{figure}[htbp]
\centering
\resizebox{0.75\textwidth}{!}{%
\includegraphics{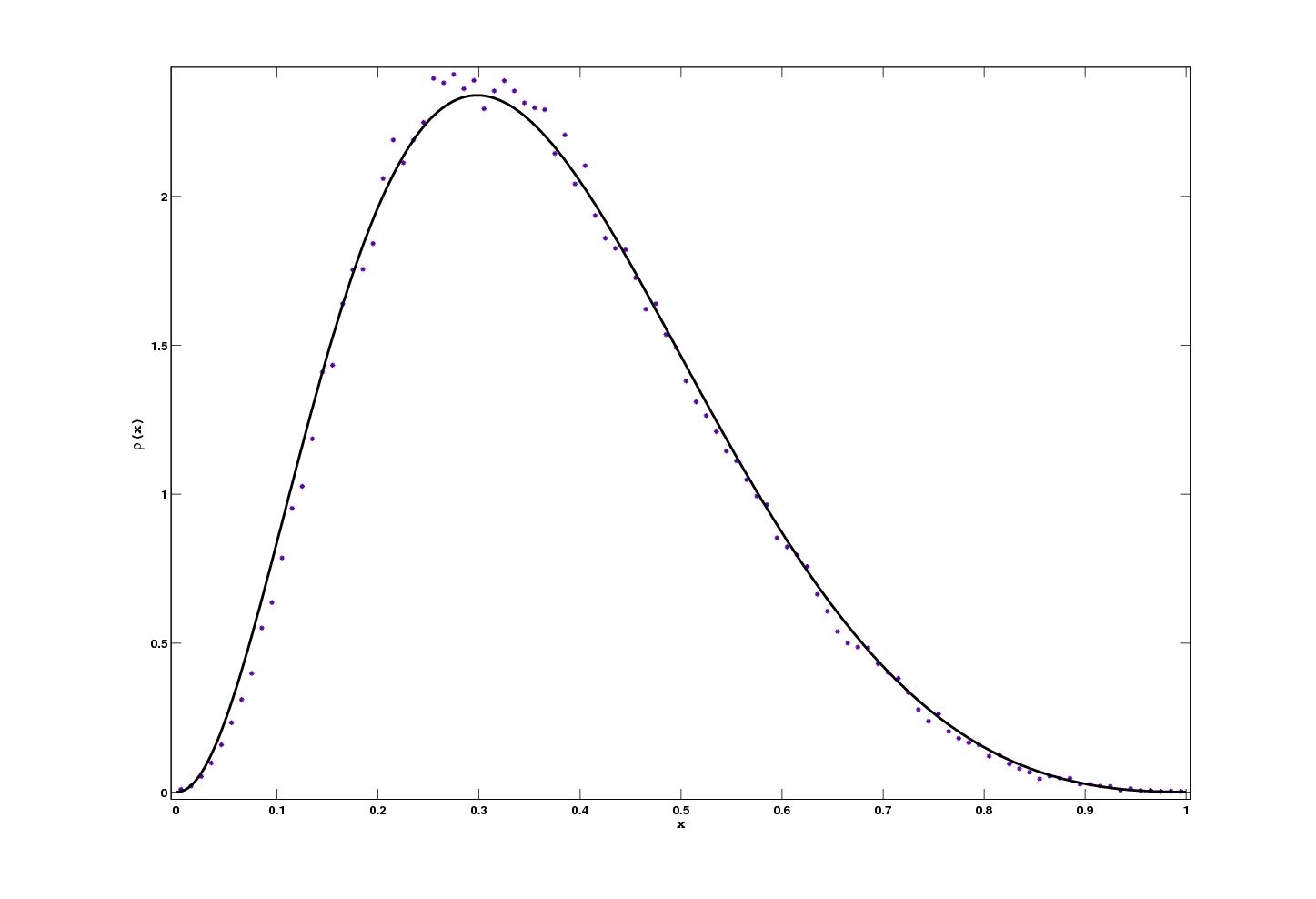}
}
%If not, use
%\vspace{5cm}       % Give the correct figure height in cm
\caption{Fit of the invariant density $\rho\left(  x\right) $ for the map $T$ with the function given in (\ref{rfit})}
\label{fig:2}       % Give a unique label
\end{figure}

An interesting question is to locate the maximum of the density $\rho.$
Numerical investigations suggest that this maximum belongs to $\left[
a_{0}^{\prime},x_{0}\right]  $ (see fig.2) depending on the parameters which
define the map $T.$

\subsection{Return times}

In section II and in section IV B of \cite{PM}, the periodic orbits of the
system, due to its invariance under $R,$ have been empirically classified by
specifying that the initial condition belongs to the half space containing,
say, the fixed point $c_{1}$ and the number of rotations they perform around
the fixed point $c_{2}$ (cfr. \cite{PM} figg. 2 and 11b).

In particular, labeling as $\Sigma_{+}$ the portion of $\Sigma$ laying in the
half space containing $c_{1},$ it can be shown by direct inspection that fig.
11b in \cite{PM}, which represents the map of the set of maximum values of $C$
in itself associated to periodic trajectories starting from $\Sigma_{+}$ after
$k$ rotations around $c_{2},$ is exactly the graph of the induced map of $T$
in the appropriate scale.

Therefore, the distribution of the number of times a trajectory of the system,
starting from $\Sigma_{+},$ winds around $c_{2}$ before hitting again
$\Sigma_{+}$ or equivalently, starting from $\Sigma_{-},$ winds around $c_{1}$
before hitting again $\Sigma_{-},$ is the same of that of the random variable
$\tau_{\left(  x_{0},1\right)  }\left(  x\right)  ,\ x\in\left(
x_{0},1\right)  ,$ being the return time on $\left(  x_{0},1\right)  $
starting from $x$ under the dynamics induced by $T.$ In terms of the already
constructed invariant measure $\mu,$ this probability is given by
\begin{align}
\mu(\tau_{(x_{0},1)}(x)  &  \geq n\ ;\ x\in(x_{0},1))=\sum_{l=n}^{\infty}%
\mu(\tau_{(x_{0},1)}(x)=l\ ;\ x\in(x_{0},1))\\
&  =\sum_{l=n}^{\infty}\mu(a_{n-2},a_{n-1})\ .\nonumber
\end{align}
But the sum on the r.h.s. can be computed using the corresponding expression
evaluated in (\ref{SL}) and we finally get
\begin{equation}
\mu(\tau_{(x_{0},1)}(x)\geq n\ ;\ x\in(x_{0},1))\approx\left(  \alpha^{\prime
}\right)  ^{-\frac{n}{B^{\ast}}}\ .
\end{equation}
We remark that the distribution of $\tau_{\left(  x_{0},1\right)  }\left(
x\right)  ,\ x\in\left(  x_{0},1\right)  ,$ is the $\mu\ a.s.$ limit of the
empirical distribution of the points appearing in fig. 2 of \cite{PM}.

We also take the occasion to remark that the average time between two crossing
of $\Sigma$ corresponds to the gap between the filled bands of points
appearing in fig.2 of \cite{PM} which has been estimated to be about $0.66\ .$
Therefore, the period of the smallest periodic orbit of the Lorenz system is
about $2\cdot0.66$ in complete agreement with what predicted by the
perturbation theory developed in \cite{Lu} and the more rigorous estimate
given in \cite{GT}.

\subsection{Statistical stability}

A slight change in the forcing term in the Lorenz equation will also change
the shape of the associated map $T$ and therefore the invariant density
associated to it, which will exist provided the perturped map still satisfies
(\ref{T_DT1}-\ref{T_DT4}). At the end of the section we will give two examples
of such a perturbation of the forcing contribution to the Lorenz field, the
first preserving the original symmetry of the Lorenz system, and the second
breaking it. As already remarked in the introduction, this last type of
perturbation has been empirically shown in \cite{CMP} to model the impact of
anthropogenic forcing to climate dynamics of the northern hemisphere as well
as the effect of the sea surface temperature on the Indian summer monsoon
rainfall variability \cite{KDC}.

Let us denote by $T_{\epsilon}$ the perturbed map. We show in this section
that under suitable assumptions the density $\rho_{\epsilon}$ of the perturbed
measure will converge to the density $\rho$ of the unperturbed one in the
$L_{m}^{1}$ norm. This kind of property is know as \emph{statistical
stability}. A former paper by Alves and Viana \cite{AV}, see also the
succesive paper by Alves \cite{Al}, addressed the question of the statistical
stability for a wide class of non-uniformly expanding maps. Their result is
based on two assumptions: (i) the perturbed map belongs to an open
neighborhood of the unperturbed one in the $C^{k}$ topology with $k\geq2$ and
(ii) the two maps are compared throughout their first return maps defined on
the \emph{same} subset where the first return maps are uniformly expanding,
with bounded distortion and long branches. Moreover, the structural parameters
of the perturbed map (especially those bounding the derivative and the
distortion) could be chosen uniformly in a $C^{k}$ neighborhood of the
unperturbed map. The main result of those papers is that when the perturbed
map converges to the unperturbed one in the $C^{k}$ topology then the density
of the absolutely continuous invariant perturbed measure converges to the
density of the unperturbed measure in the $L_{m}^{1} $ norm. Here we prove the
same result but allowing the perturbed map to be close to the unperturbed one
in the $C^{0}$ topology only. We will make use of induction but, in order to
preserve the Markov structure of the first return map, we will compare the
perturbed and the unperturbed first return maps on \emph{different} induction
subsets. The difficulty will therefore arise in the comparison of the
Perron-Frobenius operators, which will now be defined on different functional
spaces. The proof we give is inspired by the recent work \cite{BV}, but it
contains the important improvement of changing the domains of inductions.
Contrarily to \cite{AV} we are not able to establish the continuity of the map
$T_{\epsilon}\mapsto\rho_{\epsilon}$ and this is surely due to the fact that
we only require the maps $C^{0}$ close. On the other hand, discarding
regularity allows us to cover a much wider class of examples; we believe in
fact that our techniques could be used to prove the statistical stability for
general classes of maps with some sort of criticalities and
singluarities.\newline

\textbf{Assumptions on the perturbed map}

\begin{itemize}
\item[\emph{Assumption A}] $T_{\epsilon}$ is a Markov map of the unit interval
which is one-to-one and onto on the intervals $[0,x_{\epsilon,0})$ and
$(x_{\epsilon,0},1],$ convex on both sides and of class $C^{1+\iota_{\epsilon
}}$ on the open interval $(0,x_{\epsilon,0})\cup(x_{\epsilon,0},1).$

\item[\emph{Assumption B}] Let $\left\Vert \cdot\right\Vert _{0}$ denotes the
$C^{0}$-norm on the unit interval, then
\begin{equation}
\lim_{\epsilon\rightarrow0}\left\Vert T_{\epsilon}-T\right\Vert _{0}=0\ .
\end{equation}
Moreover, $\forall x\in\lbrack0,1],\ x\neq x_{0},$ we can find $\epsilon(x)$
such that,\linebreak$\forall\epsilon<\epsilon(x),\ DT_{\epsilon}$ exists and
is finite and we have
\begin{equation}
\lim_{\epsilon\rightarrow0}DT_{\epsilon}(x)=DT(x)\ .
\end{equation}
Furthermore,
\begin{equation}
\lim_{x\rightarrow x_{0}^{+}}\lim_{\epsilon\rightarrow0}\frac{DT_{\epsilon
}(x)}{DT(x)}=\lim_{x\rightarrow x_{0}^{-}}\lim_{\epsilon\rightarrow0}%
\frac{DT_{\epsilon}(x)}{DT(x)}=1\ .
\end{equation}

\item[\emph{Assumption C}] Let us denote by $C_{h,\epsilon}$ and
$\iota_{\epsilon}$ respectively the H\"{o}lder constant and the H\"{o}lder
exponent for the derivative of $T_{\epsilon}$ on the open interval
$(0,x_{\epsilon,0})\cap(x_{\epsilon,0},1);$ namely: $|DT_{\epsilon
}(x)-DT_{\epsilon}(y)|\leq C_{h,\epsilon}|x-y|^{\iota_{\epsilon}}$ for any
$x,y$ either in $(0,x_{\epsilon,0})$ or in $(x_{\epsilon,0},1).$ We assume
$C_{h,\epsilon}$ and $\iota_{\epsilon}$ to converge to the corresponding
quantities for $T$ in the limit $\epsilon\rightarrow0.$

\item[\emph{Assumption D}] Let us set $d_{(\epsilon,1,0)}:=\inf_{(b_{\epsilon
,1},a_{\epsilon,0})}|DT_{\epsilon}(x)|.$ We assume $d_{(\epsilon,1,0)}>1$ and
that there exists a constant $d_{c}$ and $\epsilon_{c}=\epsilon\left(
d_{c}\right)  $ such that,\newline$\forall\epsilon<\epsilon_{c},\ |d_{(1,0)}%
-d_{(\epsilon,1,0)}|<d_{c}.$
\end{itemize}

\emph{Remark on the notation.} To simplify the notations we will
set:\newline$W_{n}^{\prime}=(a_{n}^{\prime},a_{n-1}^{\prime});W_{n}%
=(a_{n},a_{n+1})$ and we will denote by $W_{\epsilon,n}^{\prime}%
=(a_{\epsilon,n}^{\prime},a_{\epsilon,n-1}^{\prime})$ and $W_{\epsilon
,n}=(a_{\epsilon,n},a_{\epsilon,n+1})$ the corresponding perturbed intervals,
where $a_{\epsilon,n}^{\prime}$ and $a_{\epsilon,n},\ n>1,$ are the preimages
of the maximum point $x_{\epsilon,0}.$ We also set
\begin{equation}
Z_{\epsilon,1}=Z_{\epsilon,1}^{1}\cup Z_{\epsilon,1}^{2}\ ;\ Z_{\epsilon
,1}^{1}:=(a_{\epsilon,0}^{\prime},b_{\epsilon,1}^{\prime})\ ,\ Z_{\epsilon
,1}^{2}:=(b_{\epsilon,1},a_{\epsilon,0})
\end{equation}
and
\begin{equation}
Z_{\epsilon,n}=Z_{\epsilon,n}^{1}\cup Z_{\epsilon,n}^{2}\ ;\ Z_{\epsilon
,n}^{1}:=(b_{\epsilon,n-1}^{\prime},b_{\epsilon,n}^{\prime})\ ,\ Z_{\epsilon
,n}^{2}:=(b_{\epsilon,n},b_{\epsilon,n-1}),
\end{equation}
where $T_{\epsilon}b_{\epsilon,n}^{\prime}=T_{\epsilon}b_{\epsilon
,n}=a_{\epsilon,n-1}.$ The same notation will be used for the corresponding
unperturbed intervals. We denote by $I_{\epsilon}:=(a_{\epsilon,0}^{\prime
},a_{\epsilon,0})\backslash\left\{  x_{\epsilon,0}\right\}  $ the interval
where we will induce with the first perturbed return map. From now on, we will
denote by $F$ the first return map of $T$ over $I,$ by $F_{\epsilon}$ the
first return map of $T_{\epsilon}$ on $I_{\epsilon}$ and by $P$ and
$P_{\epsilon}$ the Perron-Frobenius operators associated respectively with $F$
and $F_{\epsilon}.$ If $t\in T^{-n}z,$ where $t=T_{i_{n}}^{-1}\circ
T_{i_{n-1}}^{-1}\dots\circ T_{i_{1}}^{-1}z$ with $i_{k}=1$ or $2,$ we will
call the sequence $i_{1},\dots,i_{n}$ the \emph{signature} of $t$ relatively
to $z.$

\begin{remark}
\label{FR} The preceding assumptions imply that the order of tangency of
$T_{\epsilon}$ in $0,x_{0}$ and $1$ tends, in the limit $\epsilon
\rightarrow0,$ to that of $T.$ This is the first requirement to get again
Lemma 1 for the perturbed map. The other requirement is expressed by
Assumption D which guarantees the condition (i) in Lemma 1. Notice that this
condition cannot be deduced by assumptions (A)-(C). On the other hand, the
assumptions (ii) and (iii) of Lemma 1 are still valid for the perturbed map
since we have only to control a finite number of relations among the
corresponding derivatives. For istance, by using Assumptions B and C, we have
\begin{equation}
|DT(a_{l}^{\prime})-DT_{\epsilon}(a_{\epsilon,l}^{\prime})|\leq|DT_{\epsilon
}(a_{l}^{\prime})-DT_{\epsilon}(a_{\epsilon,l}^{\prime})|+|DT_{\epsilon}%
(a_{l}^{\prime})-DT(a_{l }^{\prime})|\ .
\end{equation}
The first term on the right hand side of this inequality is controlled by the
H\"{o}lder continuity of the derivative of $T_{\epsilon}$ while the second one
is controlled by the local convergence to $DT$ of $DT_{\epsilon}$ in the limit
$\epsilon\rightarrow0.$ However, we need some more informations for the first
return maps, which are summarized in the following Lemma.
\end{remark}

\begin{lemma}
\label{SL}

\begin{itemize}
\item[(i)] For any $n\geq0,$ let $t_{n}$ and $t_{\epsilon,n}$ two preimages of
order $n$ of $x_{0}$ and $x_{\epsilon,0}$ respectively with the same signature
with respect to these two points. Then, $\lim_{\epsilon\rightarrow
0}t_{\epsilon,n}=t_{n}.$

\item[(ii)] For any $n>0$ we have:
\begin{equation}
\lim_{\epsilon\rightarrow0}\left\Vert T_{\epsilon}^{n}-T^{n}\right\Vert
_{0}=0\,.
\end{equation}

\item[(iii)] For any $x\neq\cup_{k=0}^{\infty}T^{-k}x_{0},$ $n>0$ there exists
$\epsilon(x,n)$ such that, for any $\epsilon<\epsilon(x,n),$ the derivative
$DT_{\epsilon}^{n}(x)$ exists and is finite and moreover
\begin{equation}
\lim_{\epsilon\rightarrow0}DT_{\epsilon}^{n}(x)=DT^{n}(x)\,.
\end{equation}

\item[(iv)] For any $n\geq1,$ let $[u_{n},v_{n}],[u_{\epsilon,n}%
,v_{\epsilon,n}]\subset\left[  0,1\right]  $ such that $u_{\epsilon
,n}\rightarrow u_{n},v_{\epsilon,n}\rightarrow v_{n}$ in the limit
$\epsilon\rightarrow0$ and $T_{\epsilon}^{n}\upharpoonleft_{\lbrack
u_{\epsilon,n},v_{\epsilon,n}]},T^{n}\upharpoonleft_{\lbrack u_{n},v_{n}]}$
are injective on the respective images. Then, setting for any $y\in
T_{\epsilon}^{n}([u_{\epsilon,n},v_{\epsilon,n}])\cap T^{n}([u_{n},v_{n}]),$%
\begin{equation}
T_{\epsilon}^{-\left(  n\right)  }:=(T_{\epsilon}^{n}\upharpoonleft_{\left[
u_{\epsilon,n},v_{\epsilon,n}\right]  })^{-1},T^{-\left(  n\right)  }%
:=(T^{n}\upharpoonleft_{\left[  u_{n},v_{n}\right]  })^{-1}\ ,
\end{equation}
$T_{\epsilon}^{-\left(  n\right)  }(y)\rightarrow T^{-\left(  n\right)  }(y)$
in the limit $\epsilon\rightarrow0.$
\end{itemize}
\end{lemma}

\begin{proof}

\begin{itemize}
\item[(i)] We prove it for $n=0,$ for $n\geq1$ the proof will follow by
induction. Suppose $x_{\epsilon,0}$ does not converge to $x_{0},$ then passing
to subsequences, by compactness, there exists a subsequence $\epsilon_{n}$ and
a point $\tilde{x}\neq x_{0}$ such that $x_{\epsilon_{n},0}\rightarrow
\tilde{x}$ for $n\rightarrow\infty$. In such a point $T(\tilde{x})<1$ since
$T$ has only one maximum located at $x_{0}.$ Now, $|T_{\epsilon_{n}%
}(x_{\epsilon_{n},0})-T(\tilde{x})|=|1-T(\tilde{x})|>0. $ We now fix
$\sigma>0$ and choose $n$ large enough, depending on $\sigma,$ in such a way
that for uniform convergence we get
\begin{align}
\left\vert T_{\epsilon_{n}}(x_{\epsilon_{n},0})-T(\tilde{x})\right\vert  &
=\left\vert T_{\epsilon_{n}}(x_{\epsilon_{n},0})-T_{\epsilon_{n}}(\tilde
{x})+T_{\epsilon_{n}}(\tilde{x})+T(x_{\epsilon_{n},0})\right. \\
&  \left.  -T(x_{\epsilon_{n},0})+T(\tilde{x})\right\vert \nonumber\\
&  \leq2\left\Vert T_{\epsilon_{n}}-T\right\Vert _{0}+|T(x_{\epsilon_{n}%
,0})-T_{\epsilon_{n}}(\tilde{x})|\nonumber\\
&  \leq2\sigma+|T(x_{\epsilon_{n},0})-T_{\epsilon_{n}}(\tilde{x})|\ .\nonumber
\end{align}
In the limit $n\rightarrow\infty$ the second term on the right hand side of
the previous inequality goes to zero by assumption B and by the continuity of
$T.$ We finally send $\sigma$ to zero getting a contradiction with the above
strictly positive lower bound.

\item[(ii)] The proof is standard and by induction and it uses the uniform
continuity of $T^{n}$ on the closed unit interval.

\item[(iii)] We use induction again. Suppose the limit holds for $n.$ Then we
write
\begin{gather}
|DT_{\epsilon}^{n+1}(x)-DT^{n+1}(x)|=\\
\left\vert DT_{\epsilon}(T_{\epsilon}^{n}(x))DT_{\epsilon}^{n}(x)-DT(T^{n}%
(x))DT^{n}(x)+\right. \nonumber\\
\left.  DT_{\epsilon}(T^{n}(x))DT_{\epsilon}^{n}(x)-DT_{\epsilon}%
(T^{n}(x))DT_{\epsilon}^{n}(x)\right\vert \ .\nonumber
\end{gather}
Now, we know that: (a) $x\neq\cup_{k=0}^{n-1}T^{-k}x_{0}$ by assumption, and
also (b) $x\neq\cup_{k=0}^{n-1}T_{\epsilon}^{-k}x_{\epsilon,0}$ since by the
induction assumption the derivative $DT^{n}_{\epsilon}(x)$ is well defined at
$x\neq\cup_{k=0}^{\infty}T^{-k}x_{0}$. We need to take $\epsilon$ even
smaller, than a certain $\epsilon(x, n)$, to guarantee that $|DT_{\epsilon
}^{n+1}(x)|$ is well defined too. This is easily achieved since the preimages
of $x_{0}$ and $x_{\epsilon,0}$ converge to each other according to signature
and by choosing $\epsilon$ small enough depending on $x$ and $n$ we could just
get (a) and (b) at the same time and for $n$. We can now bound the previous
expression by:
\begin{align}
&  |DT_{\epsilon}^{n}(x)||DT_{\epsilon}(T_{\epsilon}^{n}(x))-DT_{\epsilon
}(T^{n}(x))|\\
&  +|DT_{\epsilon}(T^{n}(x))DT_{\epsilon}^{n}(x)-DT(T^{n}(x))DT^{n}%
(x)|\ .\nonumber
\end{align}
The second term converges to zero by the induction assumption. The first term
can be bounded making use of the H\"{o}lder continuity assumption on the
derivative, namely
\begin{equation}
|DT_{\epsilon}(T_{\epsilon}^{n}(x))-DT_{\epsilon}(T^{n}(x))|\leq
C_{h,\epsilon}|T_{\epsilon}^{n}(x))-T^{n}(x)|^{\iota_{\epsilon}}\ ,
\end{equation}
and of Assumption C assuring $C_{h,\epsilon}$ and $\iota_{\epsilon}$ to
converge to the corresponding quantities given for $T.$

\item[(iv)] Let us set $y_{n}:=T^{-\left(  n\right)  }(y)\in\lbrack
u_{n},v_{n}]$ and $y_{\epsilon,n}:=T_{\epsilon}^{-\left(  n\right)  }%
(y)\in\lbrack u_{\epsilon,n},v_{\epsilon,n}].$ Suppose $y_{\epsilon,n}$ does
not converge to $y_{n}.$ Then, by passing again to subsequences and by
compactness, we can find $\tilde{y}\neq y_{n}$ such that\linebreak%
$\lim_{k\rightarrow\infty}y_{\epsilon_{k},n}=\tilde{y}.$ But $y=T_{\epsilon
_{k}}^{n}(y_{\epsilon_{k},n})=T_{\epsilon_{k}}^{n}(y_{\epsilon_{k}%
,n})-T_{\epsilon_{k}}^{n}(\tilde{y})+T_{\epsilon_{k}}^{n}(\tilde{y}).$ For $k$
going to infinity the last term tends to a value different from $y$ since $T$
is injective over $[u_{n},v_{n}],$ while the first difference goes to zero by
(ii) above.
\end{itemize}
\end{proof}

It is clear that with the previous assumptions the map $T_{\epsilon}$ will
admit a unique absolutely continuous invariant measure with density
$\rho_{\epsilon}.$ This density will be related to the invariant density
$\hat{\rho}_{\epsilon}$ of the first return map $F_{\epsilon}$ on
$I_{\epsilon}$ by the formula (\ref{FF}), with normalizing constant
$C_{\epsilon,r}.$ Our next result will be to prove the statistical stability
of the unperturbed density, namely

\begin{proposition}%
\begin{equation}
\lim_{\epsilon\rightarrow0^{+}}\left\Vert \rho-\rho_{\epsilon}\right\Vert
_{L_{m}^{1}}=0\,.
\end{equation}

\end{proposition}

\begin{proof}
The proof is divided into two parts. The second part, which concerns the
comparison of the invariant densities outside the regions of induction, will
follow closely the proof of an analogous result given in \cite{BV}, but in our
case the proof will be easier since the quantities we are going to consider
have an exponential tail contrarily the corresponding ones analysed in
\cite{BV} where the presence of a neutral fixed point forced those quantities
to decay polynomially fast. The first part concerns the comparison of the
invariant densities inside the regions of induction and this part is new.

\begin{itemize}
\item[\emph{First part}] Let us suppose without restriction that the induction
sets $I=(a_{0}^{\prime},a_{0})\backslash\{x_{0}\},\newline I_{\epsilon
}=(a_{\epsilon,0}^{\prime},a_{\epsilon,0})\backslash\{x_{\epsilon,0}\}$ verify
$a_{\epsilon,0}^{\prime}<a_{0}^{\prime},\ a_{\epsilon,0}<a_{0}.$ In the
following, to ease the notation we will simply write $dx$ instead of $dm(x)$
for the (normalized) Lebesgue measure on $[0,1]$ and, for any interval
$J\subset\left[  0,1\right]  ,$ we will set $\left\vert J\right\vert
:=m\left(  J\right)  .$ We begin by bounding
\begin{equation}
\int_{I\cap I_{\epsilon}}|\hat{\rho}\left(  x\right)  -\hat{\rho}_{\epsilon
}\left(  x\right)  |dx\ .\label{E1}%
\end{equation}
In footnote 3 we defined the Banach spaces $B(I)$ and $B(I_{\epsilon}),$ which
are invariant respectively under the action of the Perron-Frobenius operators
$P$ and $P_{\epsilon}.$ The densities $\hat{\rho}$ and $\hat{\rho}_{\epsilon}$
belong respectively to these spaces and they are Lipschitz continuous on the
open intervals $(a_{0}^{\prime},x_{0})\cup(x_{0},a_{0})$ and $(a_{\epsilon
,0}^{\prime},x_{\epsilon,0})\cup(x_{\epsilon,0},a_{\epsilon,0}).$ In fact we
have to consider the action of the Perron-Frobenius operators on a larger
functional space namely that of functions of bounded variation. It is a
standard result that the Perron-Frobenius operator associated to Gibbs-Markov
maps with bounded distortion leaves invariant this space and moreover it
satisfies a Lasota-Yorke inequality for the complete norm given by the sum of
the total variation and the $L_{m}^{1}$ norm, see for instance \cite{AB} for
an account of these results. We denote by $BV(I)$ and $BV(I_{\epsilon})$ the
Banach spaces of functions of bounded variations defined respectively on the
induction sets $I$ and $I_{\epsilon},$ and by $\left\Vert \cdot\right\Vert
_{BV(I)},\ \left\Vert \cdot\right\Vert _{BV(I_{\epsilon})}$ the respective
norms. We remark that the Lebesgue measure associated to this norms should be
understood as normalized to the sets $I$ and $I_{\epsilon}.$ Since the
Perron-Frobenius operators $P$ and $P_{\epsilon}$ are quasi-compact on
respectively $BV(I)$ and $BV(I_{\epsilon}),$ we know that, in the limit
$n\rightarrow\infty,$%
\begin{align}
\left\Vert P^{n}\mathbf{1}_{I}-\hat{\rho}\right\Vert _{BV(I)} &
\rightarrow0\ ,\label{E2}\\
\left\Vert P_{\epsilon}^{n}\mathbf{1}_{I_{\epsilon}}-\hat{\rho}_{\epsilon
}\right\Vert _{BV(I_{\epsilon})} &  \rightarrow0\ .\label{E3}%
\end{align}
It will be important for what follows the convergence of the two previous
limits to be uniform with respect to $\epsilon$ in the $L_{m}^{\infty},$ and
therefore in the $L_{m}^{1},$ norms. This is guaranteed by the results in
\cite{LSV}, in particular Lemmas 4.8 and 4.11. As a matter of fact, our first
return Gibbs-Markov maps fit the assumptions of the \emph{covering systems}
with countably many branches investigated in \cite{LSV}. In particular, it can
be proven that there exist two constants $C$ an $\Lambda$ such that
$\left\Vert P^{n}\mathbf{1}_{I}-\hat{\rho}\right\Vert _{\infty}\leq
C\Lambda^{n},$ where the constant $C$ and the rate $\Lambda$ have an explicit
and $C^{\infty}$ dependence on some parameters charaterizing the map and its
expanding properties\footnote{These constants can be explicitly computed using
the Hilber metric approach. In particular $C=(1+a)D_{H}e^{D_{H}\Lambda
^{-2N_{0}}}\Lambda^{-2N_{0}},$ and $\Lambda=\left(  \tanh\frac{D_{H}}%
{4}\right)  ^{\frac{1}{N_{0}}}$. The integer $N_{0}$ insures that the
hyperbolic diameter of the iterate $P^{N_{0}}$ of a certain cone of bounded
variation functions is finite and bounded by $D_{H}.$ In particular,
$a,\Delta$ and $D_{H}$ are smooth functions of the quantities $\nu$ and $D$
entering the Lasota-Yorke inequality (see next footnote).}. Therefore, given
$\eta>0$ we can choose $n$ large enough, depending on $\eta,$ and such that
\begin{align}
\int_{I\cap I_{\epsilon}}|\hat{\rho}-\hat{\rho}_{\epsilon}|dx &  =\int_{I\cap
I_{\epsilon}}\left\vert \hat{\rho}-P^{n}\mathbf{1}_{I}+P^{n}\mathbf{1}%
_{I}+P_{\epsilon}^{n}\mathbf{1}_{I_{\epsilon}}-P_{\epsilon}^{n}\mathbf{1}%
_{I_{\epsilon}}-\hat{\rho}_{\epsilon}\right\vert dx\nonumber\\
&  \leq2\eta+\int_{I\cap I_{\epsilon}}\left\vert P^{n}\mathbf{1}%
_{I}-P_{\epsilon}^{n}\mathbf{1}_{I_{\epsilon}}\right\vert dx\ .\label{E5}%
\end{align}
Let us introduce, for $n\geq2,$%
\begin{equation}
\hat{\rho}_{n}:=P^{n-1}\mathbf{1}_{I}\ ;\ \hat{\rho}_{\epsilon,n}%
:=P_{\epsilon}^{n-1}\mathbf{1}_{I_{\epsilon}}%
\end{equation}
and finally $\tilde{\rho}_{n}:=\hat{\rho}_{n}$ on $I\cap I_{\epsilon}$ and
$\tilde{\rho}_{n}:=a_{n},$ on $I_{\epsilon}\backslash(I\cap I_{\epsilon}),$
where\linebreak$a_{n}=\lim_{x\rightarrow a_{0}^{\prime+}}\hat{\rho}_{n}(x).$
Notice that this right limit exists since $\hat{\rho}_{n}$ is Lipschitz
continuous on $(a_{0}^{\prime},x_{0})$ and moreover $\tilde{\rho}_{n}\in
BV(I_{\epsilon})$ as proven in Section 2. We remark that the need of
considering $BV\left(  I_{\epsilon}\right)  $ follows by the fact that
$\tilde{\rho}_{n}$ could be discontinuous in $x_{\epsilon,0}.$ Let us rewrite
the second term in (\ref{E5}) as
\begin{gather}
\int_{I\cap I_{\epsilon}}\left\vert P^{n}\mathbf{1}_{I}-P_{\epsilon}%
^{n}\mathbf{1}_{I_{\epsilon}}\right\vert dx=\int_{I\cap I_{\epsilon}}%
|P\hat{\rho}_{n}-P_{\epsilon}\hat{\rho}_{\epsilon,n}|dx\label{E6}\\
\leq\int_{I\cap I_{\epsilon}}\left\vert P\hat{\rho}_{n}-P_{\epsilon}%
\tilde{\rho}_{n}\right\vert dx+\int_{I\cap I_{\epsilon}}|P_{\epsilon}%
\tilde{\rho}_{n}-P_{\epsilon}\hat{\rho}_{\epsilon,n}|dx\ .\nonumber
\end{gather}
We now consider the term $\int_{I\cap I_{\epsilon}}|P_{\epsilon}\tilde{\rho
}_{n}-P_{\epsilon}\hat{\rho}_{\epsilon,n}|dx;$ by the positivity and the
contraction in $L_{m}^{1}$ of the Perron-Frobenius operator, we have
\begin{gather}
\int_{I\cap I_{\epsilon}}|P_{\epsilon}\tilde{\rho}_{n}-P_{\epsilon}\hat{\rho
}_{\epsilon,n}|dx\leq\int_{I_{\epsilon}}|P_{\epsilon}\tilde{\rho}%
_{n}-P_{\epsilon}\hat{\rho}_{\epsilon,n}|dx\leq\int_{I_{\epsilon}}|\tilde
{\rho}_{n}-\hat{\rho}_{\epsilon,n}|dx\\
\leq\int_{I_{\epsilon}\cap I}|\hat{\rho}_{n}-\hat{\rho}_{\epsilon,n}%
|dx+\int_{I_{\epsilon}\backslash(I_{\epsilon}\cap I)}|a_{n}-\hat{\rho
}_{\epsilon,n}|dx\nonumber\\
=\int_{I_{\epsilon}\cap I}|P\hat{\rho}_{n-1}-P_{\epsilon}\hat{\rho}%
_{\epsilon,n-1}|dx+\int_{I_{\epsilon}\backslash(I_{\epsilon}\cap I)}%
|a_{n}-\hat{\rho}_{\epsilon,n}|dx\nonumber\\
\leq\int_{I_{\epsilon}\cap I}|P\hat{\rho}_{n-1}-P_{\epsilon}\hat{\rho
}_{\epsilon,n-1}|dx+m(I_{\epsilon}\backslash(I_{\epsilon}\cap I))(||\hat{\rho
}_{n}||_{\infty}+||\hat{\rho}_{\epsilon,n}||_{\infty})\nonumber
\end{gather}
where the $L_{m}^{\infty}$-norm should be understood in terms of the
normalized Lebesgue measures respectively on $I$ and $I_{\epsilon}.$ But each
of these norms is bounded by the Banach norm and in particular for $\hat{\rho
}_{n}$ we have, by the Lasota-Yorke inequality,
\begin{equation}
\left\Vert \hat{\rho}_{n}\right\Vert _{\infty}\leq\left\Vert \hat{\rho}%
_{n}\right\Vert _{BV(I)}\leq\left\Vert P^{n-1}\mathbf{1}_{I}\right\Vert
_{BV(I)}\leq\nu^{n-1}\left\Vert \mathbf{1}_{I}\right\Vert _{BV(I)}%
+D\ .\label{LY}%
\end{equation}
This last quantity, for all $n$ large enough, is less than a constant $C_{2}$
and the same argument also apply to $\left\Vert \hat{\rho}_{\epsilon
,n}\right\Vert _{\infty}.$ Moreover, setting $C_{2}$ and $C_{\epsilon,2}$ the
constants bounding (\ref{LY}) in the unperturbed and perturbed case, for
$\epsilon$ sufficiently small, we have that the difference $|C_{2}%
-C_{\epsilon,2}|$ is bounded by a constant independent of $\epsilon
$\footnote{The constant $\nu<1$ and $D$ are in fact explicitly determined in
terms of the map, we defer to \cite{AB} for the details. To compare with what
stated in \cite{AB}, we need to show that there exists a power $n_{0}$ of the
first return map $F$ having the absolute value of its derivative uniformly
larger than $2.$ In the case of interest, this follows easily from the proof
of Lemma 1 by combining the Markov structure of $F$ with the lower bound for
the absolute value of its derivative which is uniformly larger than $1$ and
which is an explicit function of the parameters describing the local behavior
of the map $T,$ in particular $\alpha^{\prime},\alpha$ and $d_{\left(
1,0\right)  }.$ The quantities $\nu$ and $D$ are then functions of the lower
bound of $\left\vert DF^{n_{0}}\right\vert $ and of the constant, which we
denote by $D^{\prime},$ appearing in the Adler's condition. This last
condition is equivalent to prove that $T$ has bounded distortion and we defer
to \cite{CHMV} where the constant bounding the distortion is explicitly
determined as a function of the parameters defining the map. In the present
case, a simple inspection of the proof in \cite{CHMV} shows that such a
constant is a multiple of $d_{\left(  0,1\right)  }.$ Hence, a contribution to
$D^{\prime}$ comes from $d_{\left(  0,1\right)  },$ while the other one
\cite{CHMV} comes from the divergent behavior of the second derivative close
to the fixed point. However, in our case, the H\"{o}lder continuity assumption
on the first derivative of the map and the exponential decay of the lenght of
$Z_{i}$ makes this second contribution simply bounded by $1.$}. By setting
\begin{equation}
G_{l}:=\int_{I_{\epsilon}\cap I}|P\hat{\rho}_{l}-P_{\epsilon}\tilde{\rho}%
_{l}|dx
\end{equation}
with $l=1,\cdots,n$ and $\hat{\rho}_{1}:=\mathbf{1}_{I},\ \hat{\rho}%
_{\epsilon},1:=\mathbf{1}_{I}{}_{\epsilon}$ we have
\begin{equation}
\int_{I\cap I_{\epsilon}}|P\hat{\rho}_{n}-P_{\epsilon}\hat{\rho}_{\epsilon
,n}|dx=\sum_{l=1}^{n}G_{l}+(n-1)C_{2}m(I_{\epsilon}\backslash(I_{\epsilon}\cap
I))
\end{equation}
where $m(I_{\epsilon}\backslash(I_{\epsilon}\cap I))=O(\epsilon).$\newline In
order to compute the term $G_{l}$ we have to use the explicit structure of the
Perron-Frobenius operator. In particular we have
\begin{gather}
\int_{I_{\epsilon}\cap I}|P\hat{\rho}_{n}-P_{\epsilon}\tilde{\rho}_{n}%
|dx=\int_{I_{\epsilon}\cap I}\left\vert \sum_{i\geq1}\frac{\hat{\rho}%
_{n}(F_{i}^{-1}x)}{|DF(F_{i}^{-1}x)|}-\sum_{i\geq1}\frac{\tilde{\rho}%
_{n}(F_{\epsilon,i}^{-1}x)}{|DF_{\epsilon}(F_{\epsilon,i}^{-1}x)|}\right\vert
dx\leq\nonumber\\
\int_{I_{\epsilon}\cap I}\sum_{i\geq1}\left\vert \frac{\hat{\rho}_{n}%
(F_{i}^{-1}x)}{|DF(F_{i}^{-1}x)|}-\frac{\hat{\rho}_{n}(F_{i}^{-1}%
x)}{|DF_{\epsilon}(F_{\epsilon,i}^{-1}x)|}+\frac{\hat{\rho}_{n}(F_{i}^{-1}%
x)}{|DF_{\epsilon}(F_{\epsilon,i}^{-1}x)|}-\frac{\tilde{\rho}_{n}%
(F_{\epsilon,i}^{-1}x)}{|DF_{\epsilon}(F_{\epsilon,i}^{-1}x)|}\right\vert
dx\,.\label{FC}%
\end{gather}
Actually what we want to do is to compare the preimages of the perturbed and
of the unperturbed first return maps whose direct images are defined on
cylinders with the \emph{same} return times. This can always be done and in
particular we will consider points $x$ whose perturbed and unperturbed
preimages are both defined. At this regard, it will be enough to erase from
$I_{\epsilon}\cap I$ the open interval with endpoints $x_{\epsilon,0},x_{0}$
whose measure goes to zero in the limit $\epsilon\rightarrow0.$ We will prove
that the sum in (\ref{FC}) is bounded uniformly in $\epsilon$ in order to
exchange the sum with the limit $\epsilon\rightarrow0.$ We remind that the
perturbed and unperturbed induced first return maps are Gibbs-Markov and have
bounded distortion and that $\left\Vert \hat{\rho}_{n}\right\Vert _{\infty
}<C_{2}.$ Therefore, on each interval $Z_{i}^{j}$ (resp. $Z_{\epsilon,i}^{j}$)
$i\geq1,j=1,2,$ where $F$ (resp. $F_{\epsilon,i}$) is injective we have:

\begin{itemize}
\item For any $i\geq1;j=1,2$ and $\forall x,y\in F(Z_{i}^{j})$, we have
$\frac{|DF(F_{i}^{-1}x)|}{|DF(F_{i}^{-1}y)|}\leq D_{1}$\ and $\forall x,y\in
F(Z_{\epsilon,i}^{j})$ we have $\frac{|DF_{\epsilon}(F_{\epsilon,i}^{-1}%
x)|}{|DF_{\epsilon}(F_{\epsilon,i}^{-1}y)|}\leq D_{2}.$

\item For $\epsilon$ small enough, by the argument developed in the footnote
(4), the difference $|D_{1}-D_{2}|$ is bounded by a constant independent of
$\epsilon.$

\item There exists $y\in Z_{i}^{j}$ (resp. $Z_{\epsilon,i}^{j}$) such that
$\left\vert DF\left(  y\right)  \right\vert =\frac{|F(Z_{i}^{j})|}{\left\vert
Z_{i}^{j}\right\vert }$ (resp. $\left\vert DF_{\epsilon}\left(  y\right)
\right\vert =\frac{|F_{\epsilon}(Z_{\epsilon,i}^{j})|}{\left\vert
Z_{\epsilon,i}^{j}\right\vert }$).
\end{itemize}

This immediately implies that the first term in (\ref{FC}) is bounded by
\begin{equation}
\int_{I_{\epsilon}\cap I}\sum_{i\geq1}\frac{\hat{\rho}_{n}(F_{i}^{-1}%
x)}{|DF(F_{i}^{-1}x)|}dx\leq C_{2}D_{1}\sum_{i\geq1}\frac{|Z_{i}|}{|F(Z_{i}%
)|}\ .
\end{equation}
Similar bounds hold also for the other three terms in (\ref{FC}). We remind
that the images of the $Z_{i}$ have length $(x_{0}-a_{0}^{\prime})$ and the
sum over the $\left\vert Z_{i}\right\vert $'s gives the length of $I.$ We can
therefore take the limit $\epsilon\rightarrow0$ in (\ref{FC}). Let us consider
the first two terms in (\ref{FC}),
\begin{equation}
\sum_{i\geq1}\left\vert \frac{\hat{\rho}_{n}(F_{i}^{-1}x)}{|DF(F_{i}^{-1}%
x)|}-\frac{\hat{\rho}_{n}(F_{i}^{-1}x)}{|DF_{\epsilon}(F_{\epsilon,i}^{-1}%
x)|}\right\vert =\sum_{i\geq1}\left\vert \frac{\hat{\rho}_{n}(F_{i}^{-1}%
x)}{|DF(F_{i}^{-1}x)|}\right\vert \left\vert 1-\frac{|DF(F_{i}^{-1}%
x)|}{|DF_{\epsilon}(F_{\epsilon,i}^{-1}x)|}\right\vert \ .
\end{equation}
We can bound this quantity making use of Lemma ({\ref{SL}) part (iii) and part
(iv) first and then by observing that the point $F_{i}^{-1}x$ does not
coincide with $x_{0}.$ Let us set $w:=F_{i}^{-1}x;\ w_{\epsilon}%
:=F_{\epsilon,i}^{-1}x$ and $F=T^{i}.$ Then,
\begin{equation}
\left\vert \frac{DF(F_{i}^{-1}x)}{DF_{\epsilon}(F_{\epsilon,i}^{-1}%
x)}\right\vert =\prod_{m=0}^{i-1}\left\vert \frac{DT(T^{m}w)DT_{\epsilon
}(T^{m}w)}{DT_{\epsilon}(T_{\epsilon}^{m}w_{\epsilon})DT_{\epsilon}(T^{m}%
w)}\right\vert \label{DD}%
\end{equation}
We notice that }$\left\vert {T^{m}w-T_{\epsilon}^{m}w_{\epsilon}}\right\vert
=O\left(  {\epsilon}\right)  ,${\ the intervals with endpoints $T^{i}w$ and
$T_{\epsilon}^{i}w_{\epsilon}$ do not contain $x_{\epsilon,0}$ and their
length tends to zero when $\epsilon$ vanishes. Therefore,
\begin{equation}
\prod_{m=0}^{i-1}\left\vert \frac{DT(T^{m}w)}{DT_{\epsilon}(T^{m}%
w)}\right\vert \exp\left[  \sum_{m=0}^{i-1}\frac{1}{|DT_{\epsilon}%
(y)|}C_{h,\epsilon}(\left\Vert T^{m}-T_{\epsilon}^{m}\right\Vert _{0}%
^{\iota_{\epsilon}}+|T^{m}w-T^{m}{w_{\epsilon}}|)\right]  \ ,
\end{equation}
where $y$ is a point between $T^{i}w$ and $T_{\epsilon}^{i}w_{\epsilon}.$
Hence, by Assumption B, this term tends to $1$ in the limit $\epsilon
\rightarrow0$\footnote{Actually $y$ depends on $\epsilon,\ y=y_{\epsilon}.$
Setting $y^{\ast}:=\lim_{\epsilon\rightarrow0}y_{\epsilon},$ we get
\[
|DT_{\epsilon}(y_{\epsilon})-DT(y^{\ast})|\leq|DT_{\epsilon}(y_{\epsilon
})-DT_{\epsilon}(y^{\ast})|+|DT_{\epsilon}(y^{\ast})-DT(y^{\ast})|\ .
\]
\par
The first term on the r.h.s. can be bounded making use of the H\"{o}lder
continuity assumption on $DT_{\epsilon},$ the second making use of Assumption
B.}.\newline Moreover, the other couple of terms in (\ref{FC}),
\begin{equation}
\sum_{i\geq1}\left\vert \frac{\hat{\rho}_{n}(F_{i}^{-1}x)}{|DF_{\epsilon
}(F_{\epsilon,i}^{-1}x)|}-\frac{\tilde{\rho}_{n}(F_{\epsilon,i}^{-1}%
x)}{|DF_{\epsilon}(F_{\epsilon,i}^{-1}x)|}\right\vert \leq\sum_{i\geq1}%
\frac{1}{|DF_{\epsilon}(F_{\epsilon,i}^{-1}x)|}|\hat{\rho}_{n}(F_{i}%
^{-1}x)-\tilde{\rho}_{n}(F_{\epsilon,i}^{-1}x)|\ .
\end{equation}
We remark that the function $\tilde{\rho}_{n}$ is a continuous extension of
$\hat{\rho}$ to}\linebreak{$I_{\epsilon}\backslash(I\cap I_{\epsilon})$ and
therefore we can rewrite }%
\begin{equation}
{|\hat{\rho}_{n}(F_{i}^{-1}x)-\tilde{\rho}_{n}(F_{\epsilon,i}^{-1}%
x)|=|\tilde{\rho}_{n}(F_{i}^{-1}x)-\tilde{\rho}_{n}(F_{\epsilon,i}^{-1}x)|\ ,}%
\end{equation}
{where $\tilde{\rho}_{n}$ is now defined on $I\cup I_{\epsilon}.$ This
function is continuous on}\linebreak{$I\cup I_{\epsilon}\backslash\{x_{0}\}$
and, by part (iv) of Lemma (\ref{SL}),}\newline{$\lim_{\epsilon\rightarrow
0^{+}}|\tilde{\rho}_{n}(F_{i}^{-1}x)-\tilde{\rho}_{n}(F_{\epsilon,i}%
^{-1}x)|=0.$\newline To resume: for $n$ larger than a certain $n(\eta),$%
\begin{equation}
(\ref{E5})\leq2\eta+\sum_{l=1}^{n}G_{l}+(n-1)O(\epsilon)\ ;
\end{equation}
each $G_{l}$ is bounded uniformly w.r.t. $\epsilon$ and tends to zero for
$\epsilon$ tending to zero. Therefore we can pass to the limit $\epsilon
\rightarrow0$ and then $\eta\rightarrow0.$}

\item[\emph{Second part}] According to the assumptions made at the beginning
of the first part and without loss of generality, we will assume that all
$W_{\epsilon,n}$ lie to the left of the corresponding $W_{n}.$
Therefore we have
\begin{gather}
\int_{\lbrack0,1]}|\rho-\rho_{\epsilon}|dx=\int_{I\cap I_{\epsilon}}|\rho
-\rho_{\epsilon}|dx+\int_{I\cap W_{\epsilon,1}}|\rho-\rho_{\epsilon
}|dx\nonumber\\
+\int_{I_{\epsilon}\cap W_{1}^{\prime}}|\rho-\rho_{\epsilon}|dx\nonumber\\
+\sum_{l=1}^{\infty}\left\{  \int_{W_{l}\cap W_{\epsilon,l}}|\rho
-\rho_{\epsilon}|dx+\int_{W_{l}\backslash(W_{l}\cap W_{\epsilon,l})}|\rho
-\rho_{\epsilon}|dx\right\}  \nonumber\\
+\sum_{l=1}^{\infty}\left\{  \int_{W_{l}^{\prime}\cap W_{\epsilon,l}^{\prime}%
}|\rho-\rho_{\epsilon}|dx+\int_{W_{l}^{\prime}\backslash(W_{l}^{\prime}\cap
W_{\epsilon,l}^{\prime})}|\rho-\rho_{\epsilon}|dx\right\}  \ .\label{maino}%
\end{gather}
The densities are given in terms of the corresponding densities of the induced
subsets and of the multiplicative constants $C_{r}$ and $C_{\epsilon,r}.$
Hence, we should first compare the latter. Since they are surely smaller than
$1,$ we have
\begin{equation}
|C_{r}-C_{\epsilon,r}|\leq\sum_{i=1}^{\infty}i\left\vert \int_{Z_{i}^{1}}%
\hat{\rho}\frac{dx}{m(I)}-\int_{Z_{\epsilon,i}^{1}}\hat{\rho_{\epsilon}}%
\frac{dx}{m(I_{\epsilon})}\right\vert \ .
\end{equation}
The same bound holds also choosing $Z_{i}^{2}$ ($Z_{\epsilon,i}^{2}$) instead
of $Z_{i}^{1}$ ($Z_{\epsilon,i}^{1}$). The sum converges uniformly as a
function of $\epsilon$ since the $L_{m}^{\infty}$ norms of $\hat{\rho}$ and
$\hat{\rho_{\epsilon}}$ are bounded by $C_{2}$ and the lengths of the
$Z_{i}^{1}$ and $Z_{\epsilon,i}^{1}$ decay exponentially fast. We now show
that passing to the limit $\epsilon\rightarrow0$ inside the sum this vanishes.
At this regard we rewrite the previous bound as
\begin{align}
&  \sum_{i=1}^{\infty}i\left\vert \int_{Z_{i}^{1}\cap Z_{\epsilon,i}^{1}}%
\hat{\rho}\frac{dx}{m(I)}+\int_{Z_{i}^{1}\backslash(Z_{i}^{1}\cap
Z_{\epsilon,i}^{1})}\hat{\rho}\frac{dx}{m(I)}\right.  \\
&  \left.  -\int_{Z_{\epsilon,i}^{1}\cap Z_{i}^{1}}\hat{\rho_{\epsilon}}%
\frac{dx}{m(I_{\epsilon})}-\int_{Z_{\epsilon,i}^{1}\backslash(Z_{i}^{1}\cap
Z_{\epsilon,i}^{1})}\hat{\rho_{\epsilon}}\frac{dx}{m(I_{\epsilon})}\right\vert
\nonumber\\
&  \leq\sum_{i=1}^{\infty}i\left[  2C_{2}m(Z_{i}^{1}\Delta Z_{\epsilon,i}%
^{1})+C_{2}\left\vert \frac{1}{m(I)}-\frac{1}{m(I_{\epsilon})}\right\vert
+\int_{Z_{i}^{1}\cap Z_{\epsilon,i}^{1}}|\hat{\rho}-\hat{\rho_{\epsilon}%
}|\frac{dx}{m(I_{\epsilon})}\right]  \ .\nonumber
\end{align}
Each term in the last sum vanishes in the limit $\varepsilon\rightarrow0,$ in
particular the third term tends to zero by what stated in the first part of
the proof.\newline Moreover, by Lemma 1 and by the fact that the derivatives
of the maps $T$ and $T_{\epsilon}$ are strictly expanding in the neighborhood
of $x_{0},$ for $x\in(0,a_{0}^{\prime}),$ we get
\begin{equation}
\sum_{m=2}^{\infty}\sum_{l=1,2}\frac{1}{|DT^{m}(T_{l}^{-1}T_{2}^{-1}%
T_{1}^{-(m-2)}x)|}\leq C_{3}\frac{1}{\left(  \alpha\alpha^{\prime}\log
\alpha^{\prime}\right)  }:=C_{4}\ .
\end{equation}
Furthermore, for $x\in(0,a_{0}),$%
\begin{equation}
\sum_{l=1,2}\frac{1}{|DT(T_{l}^{-1}x)|}\leq\left(  \min_{(b_{2},b_{1}%
)\cup(b_{1}^{\prime},b_{2}^{\prime})}|DT|\right)  ^{-1}:=C_{5}\ .
\end{equation}
Analogous bounds hold also for the perturbed map, so we can choose the
constants $C_{4},C_{5}$ independent of $\epsilon.$ Let us call $\rho_{s}$
(resp. $\rho_{r}$), the representations of the invariant density on
$(0,x_{0})$ (resp. $(x_{0},1))$ without the normalizing factor $C_{r}.$ By the
previous bounds on the derivatives of $T$ and the boundness of the densities
on the induced spaces, it follows immediately that there exists a constant
$C_{6}$ such that the $L_{m}^{\infty}$ norms of $\rho_{s}$ and $\rho_{r}$ are
bounded by $C_{6}.$ The same argument also holds for $\rho_{\epsilon,s}$ and
$\rho_{\epsilon,r}$ and, since $C_{6}$ can be chosen independent of
$\epsilon,$ $\left\Vert \rho_{\epsilon,s}\right\Vert _{\infty},\left\Vert
\rho_{\epsilon,r}\right\Vert _{\infty}\leq C_{6}.$\newline We can now proceed
to bound each term in (\ref{maino}). For the first one we get
\begin{equation}
\int_{I\cap I_{\epsilon}}|\rho-\rho_{\epsilon}|dx\leq|C_{r}-C_{\epsilon
,r}|\int_{I\cap I_{\epsilon}}\hat{\rho}dx+C_{\epsilon,r}\int_{I\cap
I_{\epsilon}}|\hat{\rho}-\hat{\rho_{\epsilon}}|dx\ ,
\end{equation}
which can be bounded uniformly in $\epsilon$ by arguing as in the previous
computations. For the second term (the third one can be bounded in the same
way) we have
\begin{equation}
\int_{I\cap W_{\epsilon,1}}|\rho-\rho_{\epsilon}|dx\leq|C_{r}-C_{\epsilon
,r}|\int_{I\cap W_{\epsilon,1}}\hat{\rho}dx+C_{\epsilon,r}\int_{I\cap
W_{\epsilon,1}}|\hat{\rho}-\rho_{\epsilon,s}|dx\ .
\end{equation}
The right hand side is uniformly bounded in $\epsilon,$ in particular
\begin{equation}
C_{\epsilon,r}\int_{I\cap W_{\epsilon,1}}|\hat{\rho}-\rho_{\epsilon,s}%
|dx\leq(C_{2}+C_{6})m(I\cap W_{\epsilon,1})\ ,
\end{equation}
where vanishes $m(I\cap W_{\epsilon,1})$ in the limit $\epsilon\rightarrow
0.$\newline We now consider the last sum in (\ref{maino}). Similar arguments
allow to bound the remaining sum which is even easier to handle. We first
have
\begin{gather}
\sum_{l=1}^{\infty}\int_{W_{l}^{\prime}\backslash(W_{l}^{\prime}\cap
W_{\epsilon,l}^{\prime})}|\rho-\rho_{\epsilon}|dx\leq\sum_{l=1}^{\infty
}\left[  |C_{r}-C_{\epsilon,r}|\int_{W_{l}^{\prime}\backslash(W_{l}^{\prime
}\cap W_{\epsilon,l}^{\prime})}\rho_{s}dx\right.  \\
\left.  +C_{\epsilon,r}\int_{W_{l}^{\prime}\backslash(W_{l}^{\prime}\cap
W_{\epsilon,l}^{\prime})}|\rho_{s}-\rho_{\epsilon,s}|dx\right]  \ .\nonumber
\end{gather}
The sum is uniformly convergent as a function of $\epsilon$ since\linebreak%
$W_{l}^{\prime}\backslash(W_{l}^{\prime}\cap W_{\epsilon,l}^{\prime})\subset
W_{l}^{\prime}$ and the length of such an interval decays exponentially fast
with rate independent of $\epsilon.$ Finally, previous considerations imply
that each term into the sum goes to zero in the limit $\epsilon\rightarrow
0.$\newline Finally we have
\begin{gather}
\sum_{l=1}^{\infty}\int_{W_{l}^{\prime}\cap W_{\epsilon,l}^{\prime}}|\rho
-\rho_{\epsilon}|dx\leq\\
\sum_{l=1}^{\infty}\left[  |C_{r}-C_{\epsilon,r}|\int_{W_{l}^{\prime}\cap
W_{\epsilon,l}^{\prime}}\rho_{s}dx+C_{\epsilon,r}\int_{W_{l}^{\prime}\cap
W_{\epsilon,l}^{\prime}}|\rho_{s}-\rho_{\epsilon,s}|dx\right]  \ .\nonumber
\end{gather}
The preceding considerations also apply to the first sum in this formula
proving this to vanish in the limit $\epsilon\rightarrow0.$ For the second sum
we make use of the representations of $\rho_{s}$ and $\rho_{\epsilon,s}$ in
terms of the density on the induced space. Thus we have
\begin{gather}
\int_{W_{l}^{\prime}\cap W_{\epsilon,l}^{\prime}}|\rho_{s}-\rho_{\epsilon
,s}|dx\leq\\
\int_{W_{l}^{\prime}\cap W_{\epsilon,l}^{\prime}}\sum_{p=l+2}^{\infty}%
\sum_{k=1,2}\left\vert \frac{\hat{\rho}(T_{k}^{-1}T_{2}^{-1}T_{1}%
^{-(p-l-2)}x)}{|DT^{p-l}(T_{k}^{-1}T_{2}^{-1}T_{1}^{-(p-l-2)}x)|}-\frac
{\hat{\rho}_{\epsilon}(T_{\epsilon,k}^{-1}T_{\epsilon,2}^{-1}T_{\epsilon
,1}^{-(p-l-2)}x)}{|DT_{\epsilon}^{p-l}(T_{\epsilon,k}^{-1}T_{\epsilon,2}%
^{-1}T_{\epsilon,1}^{-(p-l-2)}x)|}\right\vert dx\leq\nonumber\\
\int_{W_{l}^{\prime}\cap W_{\epsilon,l}^{\prime}}\sum_{p=l+2}^{\infty}%
\sum_{k=1,2}\left\vert \frac{\hat{\rho}(T_{k}^{-1}T_{2}^{-1}T_{1}%
^{-(p-l-2)}x)}{|DT^{p-l}(T_{k}^{-1}T_{2}^{-1}T_{1}^{-(p-l-2)}x)|}-\frac
{\hat{\rho}_{\epsilon}(T_{\epsilon,k}^{-1}T_{\epsilon,2}^{-1}T_{\epsilon
,1}^{-(p-l-2)}x)}{|DT^{p-l}(T_{k}^{-1}T_{2}^{-1}T_{1}^{-(p-l-2)}%
x)|}\right\vert dx+\nonumber\\
\int_{W_{l}^{\prime}\cap W_{\epsilon,l}^{\prime}}\sum_{p=l+2}^{\infty}%
\sum_{k=1,2}\left\vert \frac{\hat{\rho}_{\epsilon}(T_{\epsilon,k}%
^{-1}T_{\epsilon,2}^{-1}T_{\epsilon,1}^{-(p-l-2)}x)}{|DT^{p-l}(T_{k}^{-1}%
T_{2}^{-1}T_{1}^{-(p-l-2)}x)|}-\frac{\hat{\rho}_{\epsilon}(T_{\epsilon,k}%
^{-1}T_{\epsilon,2}^{-1}T_{\epsilon,1}^{-(p-l-2)}x)}{|DT_{\epsilon}%
^{p-l}(T_{\epsilon,k}^{-1}T_{\epsilon,2}^{-1}T_{\epsilon,1}^{-(p-l-2)}%
x)|}\right\vert dx=\nonumber\\
Q_{1,l}+Q_{2,l}\ .\nonumber
\end{gather}
We further decompose $Q_{1,l}$ as
\begin{align}
Q_{1,l} &  =\int_{W_{l}^{\prime}\cap W_{\epsilon,l}^{\prime}}\sum
_{p=l+2}^{\infty}\sum_{k=1,2}\frac{1}{|DT^{p-l}(T_{k}^{-1}T_{2}^{-1}%
T_{1}^{-(p-l-2)}x)|}\left\vert \hat{\rho}(T_{k}^{-1}T_{2}^{-1}T_{1}%
^{-(p-l-2)}x)-\right.  \\
&  \left.  \hat{\rho}(T_{\epsilon,k}^{-1}T_{\epsilon,2}^{-1}T_{\epsilon
,1}^{-(p-l-2)}x)+\hat{\rho}(T_{\epsilon,k}^{-1}T_{\epsilon,2}^{-1}%
T_{\epsilon,1}^{-(p-l-2)}x)-\hat{\rho}_{\epsilon}(T_{\epsilon,k}%
^{-1}T_{\epsilon,2}^{-1}T_{\epsilon,1}^{-(p-l-2)}x)\right\vert \ .\nonumber
\end{align}
Changing variables, setting $y_{k}:=T_{k}^{-1}T_{2}^{-1}T_{1}^{-(p-l-2)}x$
and\linebreak$y_{\epsilon,k}:=y_{\epsilon}(y_{k})=T_{\epsilon,k}%
^{-1}T_{\epsilon,2}^{-1}T_{\epsilon,1}^{-(p-l-2)}(T^{p}y_{k}),$ since $y_{k}$
and $y_{\epsilon,k}$ belong to $Z_{p}^{k}\cup Z_{\epsilon,p}^{k},$ we get
\begin{equation}
Q_{1,l}=\sum_{k=1,2}\sum_{p=l+2}^{\infty}\int_{Z_{p}^{k}}|\hat{\rho}%
(y_{k})-\hat{\rho}(y_{\epsilon,k})+\hat{\rho}(y_{\epsilon,k})-\hat{\rho
}_{\epsilon}(y_{\epsilon,k})|dy_{k}\ .
\end{equation}
But,
\begin{equation}
\sum_{l=1}^{\infty}Q_{1,l}\leq4C_{2}\sum_{l=1}^{\infty}\sum_{p=l+2}^{\infty
}m(Z_{p})\ ,
\end{equation}
which is clearly convergent because the measure of $Z_{p}$ is exponentially
decreasing. Moreover, by what has been shown in the first part of the proof,
\begin{equation}
\lim_{\epsilon\rightarrow0}\int_{Z_{p}^{k}}|\hat{\rho}(y_{\epsilon,k}%
)-\hat{\rho}_{\epsilon}(y_{\epsilon,k})|dy_{k}=0\ .
\end{equation}
On the other hand, we first take $\epsilon$ small enough to get $y_{\epsilon
,k}$ on the same side of $x_{0}$ as $y_{k}$ and then we use the Lipschitz
continuity property of $\hat{\rho}$ to conclude, by observing that
$y_{\epsilon,k}$ tends to $y_{k}$ when $\epsilon$ tends to zero, that also
\begin{equation}
\lim_{\epsilon\rightarrow0}\int_{Z_{p}^{k}}|\hat{\rho}(y_{k})-\hat{\rho
}(y_{\epsilon,k})|dy_{k}=0\ .
\end{equation}
We now consider $Q_{2,l}$ and show that it is uniformly bounded in $\epsilon.$
As a matter of fact,
\begin{align}
Q_{2,l} &  =\int_{W_{l}^{\prime}\cap W_{\epsilon,l}^{\prime}}\sum
_{p=l+2}^{\infty}\sum_{k=1,2}\left\vert \frac{\hat{\rho}_{\epsilon
}(T_{\epsilon,k}^{-1}T_{\epsilon,2}^{-1}T_{\epsilon,1}^{-(p-l-2)}%
x)}{\left\vert DT^{p-l}(T_{k}^{-1}T_{2}^{-1}T_{1}^{-(p-l-2)}x)\right\vert
}\right\vert \times\\
&  \times\left\vert 1-\frac{\left\vert DT^{p-l}(T_{k}^{-1}T_{2}^{-1}%
T_{1}^{-(p-l-2)}x)\right\vert }{\left\vert DT_{\epsilon}^{p-l}(T_{\epsilon
,k}^{-1}T_{\epsilon,2}^{-1}T_{\epsilon,1}^{-(p-l-2)}x)\right\vert }\right\vert
dx\ .\nonumber
\end{align}
We bound it by the sum of its two parts: the density will have bounded
infinity norm; the sums in $p$ over the inverses of the derivatives are
bounded by a constant since the derivatives decay exponentially fast and the
sums over $l$ will be controlled by the measure of $W_{l}^{\prime}.$ Finally
the same arguments that led to bound (\ref{DD}) apply also to the second
factor in the previous expression proving it tends to zero in the limit
$\epsilon\rightarrow0.$
\end{itemize}

This concludes the proof.
\end{proof}

\bigskip

We end up our analysis considering two examples of the perturbed Lorenz system
giving rise to perturbed versions of the map $T$ of the kind discussed in this section.

\begin{example}
Let us consider a perturbation of the Lorenz field (\ref{l1}) obtained by
adding the constant forcing field $\left(  0,0,-\epsilon\beta\left(
\rho+\sigma\right)  \right)  ,\ \epsilon>0.$ The perturbation is easily seen
to preserve the symmetry under the involution $R$ of the unperturbed field.
Arguing as in the first section, for $\epsilon$ sufficiently small, the
perturbed system will keep the same features of the unperturbed one, hence map
$T_{\epsilon}$ is easily seen to satisfy (\ref{T_DT1}-\ref{T_DT4}) as well as
Assumptions A-D. Here it follows the plot of $T_{\epsilon}$, for $\epsilon=0.5$,
and the plot of the fit of the invariant density $\rho_{R}$ for the evolution
under the maps $T$ and $T_{\epsilon}$, corresponding respectively to the choice 
of the Poincar\'{e} surfaces $\Sigma_{+},\Sigma_{+}^{\epsilon}$, the last one being 
contructed as in the unperturbed case.
\end{example}

\begin{figure}[htbp]
\centering
\resizebox{0.75\textwidth}{!}{%
\includegraphics{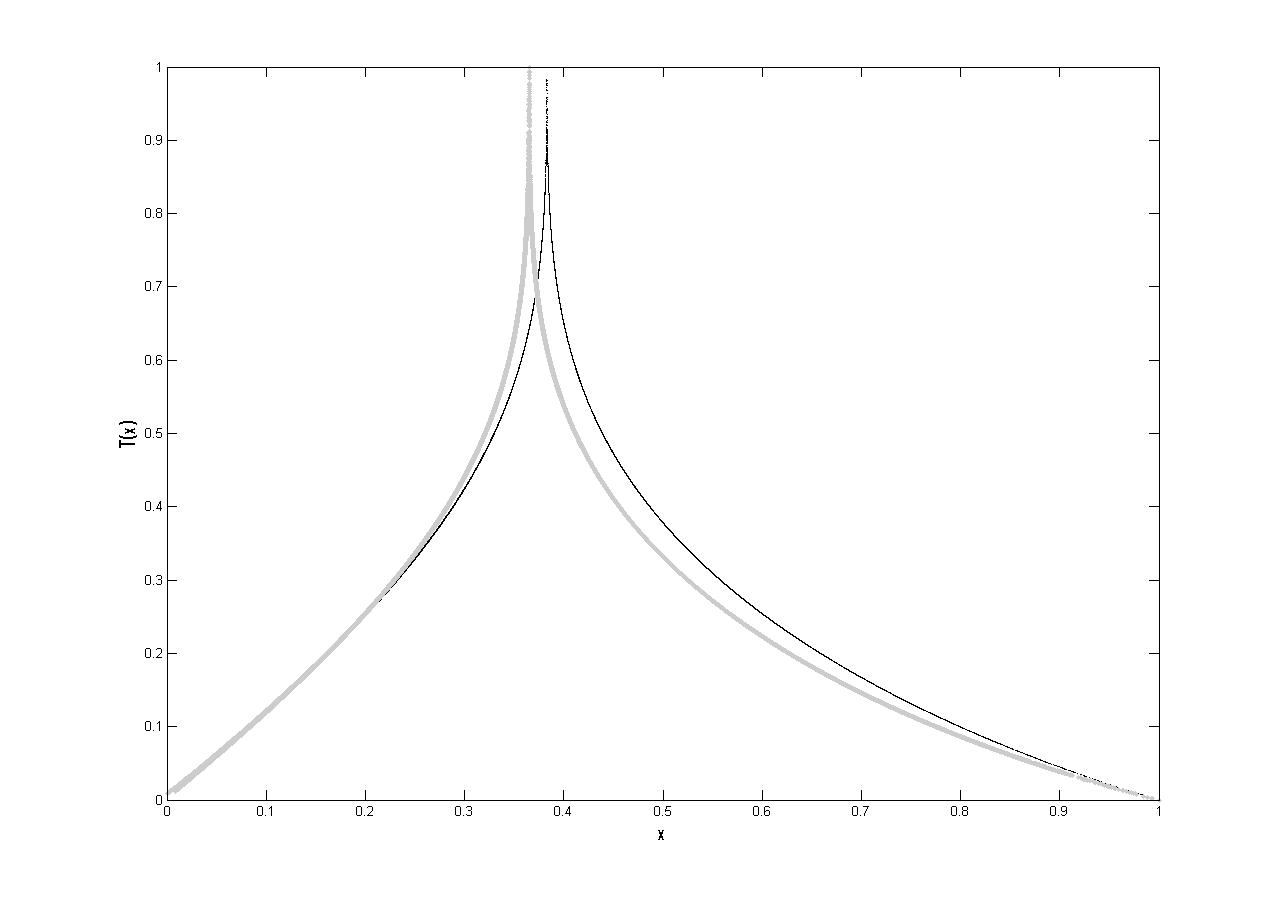}
}
%If not, use
%\vspace{5cm}       % Give the correct figure height in cm
\caption{$T_{\epsilon}$ (thick line), $T$ (thin line).}
\label{fig:3}       % Give a unique label
\end{figure}

\begin{figure}[htbp]
\centering
\resizebox{0.75\textwidth}{!}{%
\includegraphics{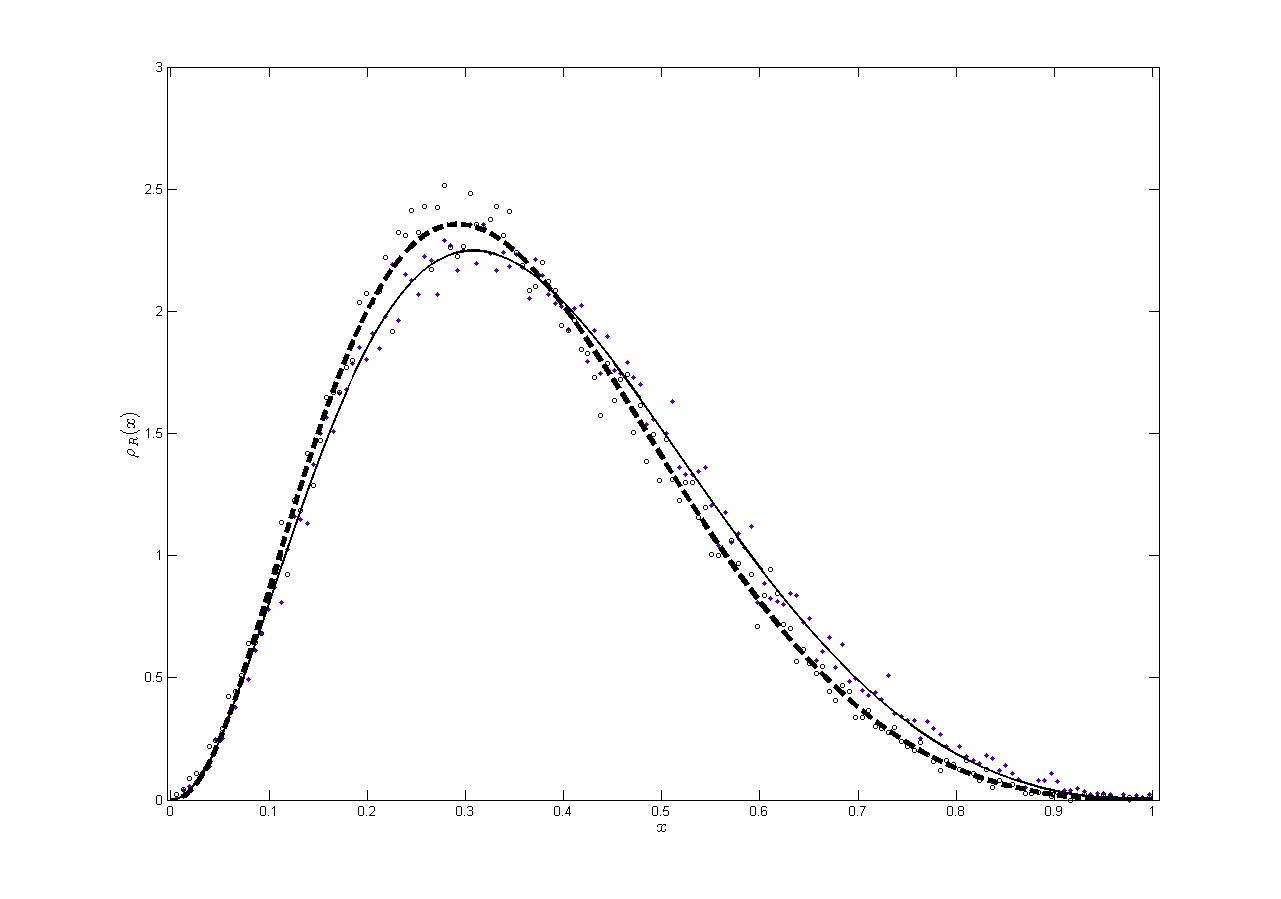}
}
%If not, use
%\vspace{5cm}       % Give the correct figure height in cm
\caption{Fit of the invariant density for the evolution under $T$ (solid line) and under $T_{\epsilon},\epsilon=0.5$, (dashed line).}
\label{fig:4}       % Give a unique label
\end{figure}

\newpage

\begin{example}
We now consider the following perturbation of the Lorenz field (\ref{l1})
realized by adding the field $\left(  \epsilon\cos\theta,\epsilon\sin
\theta,0\right)  $ where $\epsilon>0$ and $\theta\in\lbrack0,2\pi).$ The
perturbed system is not $R$-invariant anymore, anyway, for $\epsilon$
sufficiently small, the system will still have a saddle fixed point
$c_{0}^{\epsilon}$ and two unstable fixed points $c_{1}^{\epsilon}%
,c_{2}^{\epsilon}.$ Hence, for any $\theta\in\lbrack0,2\pi),$ we have two
different $T_{\epsilon},$ namely $T_{\epsilon}^{+},T_{\epsilon}^{-},$ both
satisfying (\ref{T_DT1}-\ref{T_DT4}) as well as Assumptions A-D corresponding
respectively to the choice of the Poincar\'{e} surfaces $\Sigma_{+}^{\epsilon
},\Sigma_{-}^{\epsilon},$ which can be constructed as in the unperturbed case.
To obtain meaningful plots of the deviation from $T$ of the perturbed maps as
well as of the deviation of the associated invatiant densities from the
unperturbed one, $\epsilon$ has been set equal to $2.5$ and $\theta$ to $70%
%TCIMACRO{\U{b0}}%
%BeginExpansion
{{}^\circ}%
%EndExpansion
$ as in \cite{CMP}.
\end{example}

\begin{figure}[htbp]
\centering
\resizebox{0.75\textwidth}{!}{%
\includegraphics{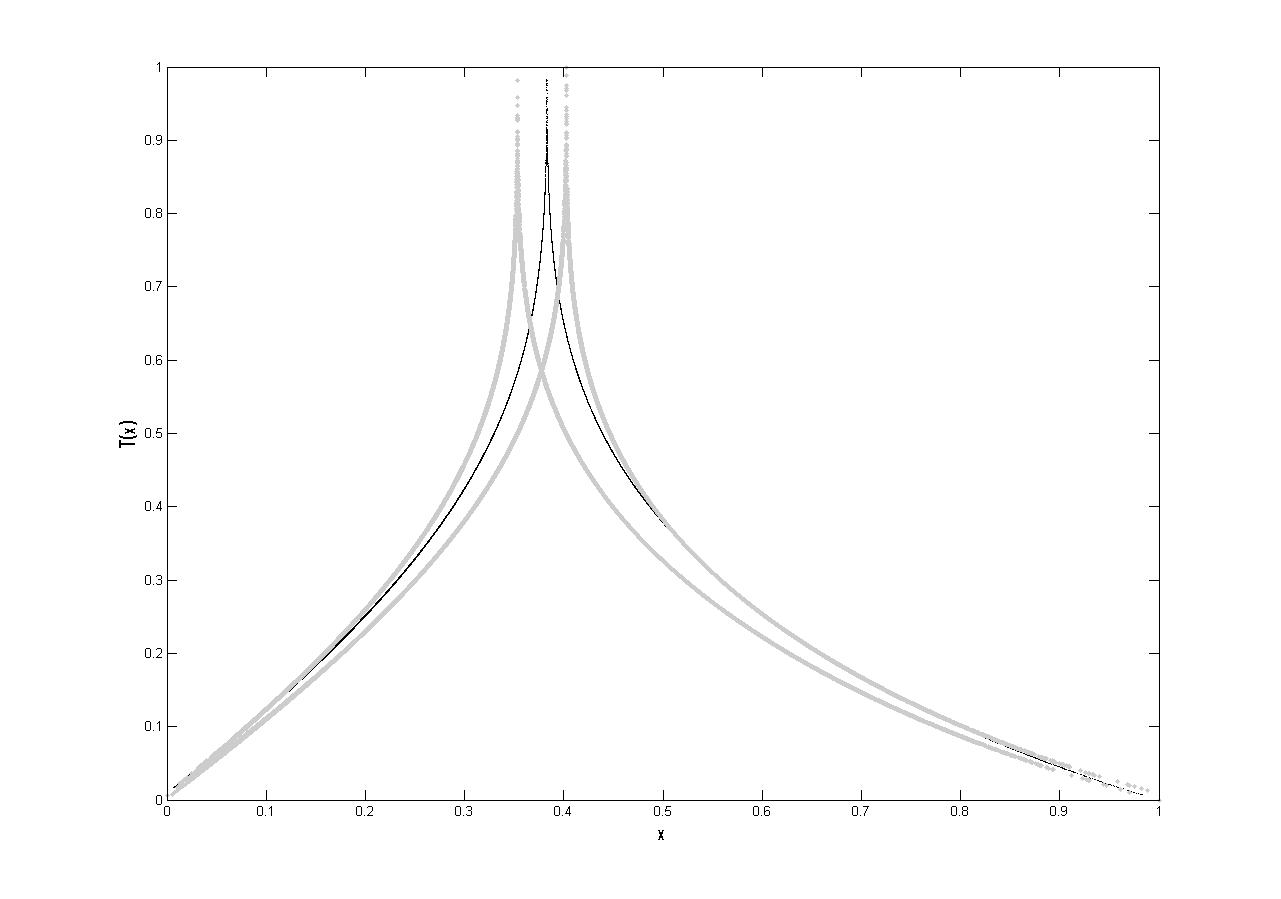}
}
%If not, use
%\vspace{5cm}       % Give the correct figure height in cm
\caption{$T$ (thin line), $T_{\epsilon}^{+}$ (thick line to the left of $T$), $T_{\epsilon}^{-}$ (thick line to the right of $T$).}
\label{fig:5}       % Give a unique label
\end{figure}

\begin{figure}[htbp]
\centering
\resizebox{0.75\textwidth}{!}{%
\includegraphics{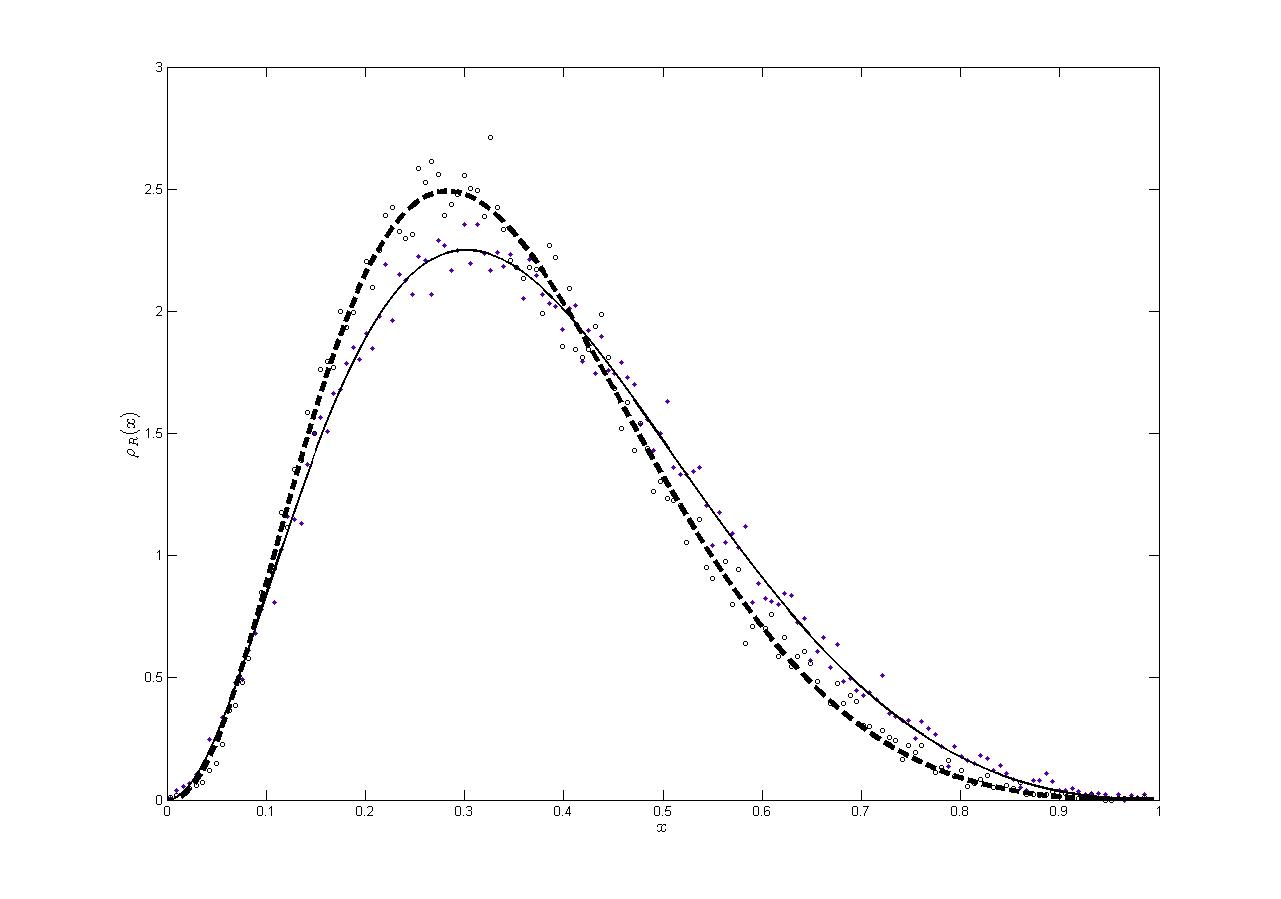}
}
%%If not, use
%%\vspace{5cm}       % Give the correct figure height in cm
\caption{Fit of the invariant densities for the evolution under $T$ (solid line) and under $T_{\epsilon}^{+}$ (dashed line).}
\label{fig:6}       % Give a unique label
\end{figure}

\begin{figure}[htbp]
\centering
\resizebox{0.75\textwidth}{!}{%
\includegraphics{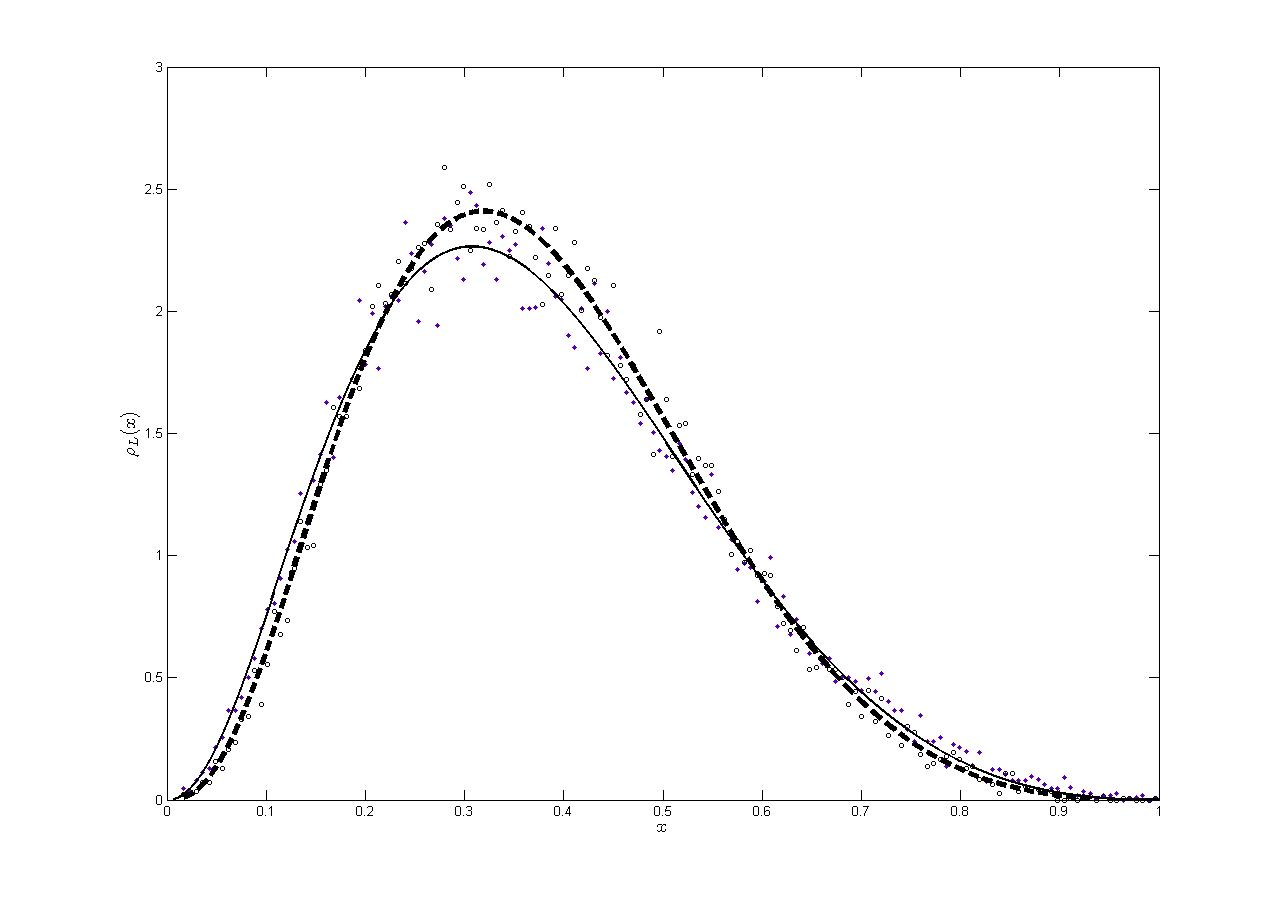}
}
%If not, use
%\vspace{5cm}       % Give the correct figure height in cm
\caption{Fit of the invariant densities for the evolution under $T$ (solid line) and under $T_{\epsilon}^{-}$ (dashed line).}
\label{fig:7}       % Give a unique label
\end{figure}

\begin{itemize}
\item Michele GIANFELICE

Dipartimento di Matematica

Universit\`{a} della Calabria

Campus di Arcavacata

Ponte P. Bucci - Cubo 30B

I-87036 Arcavacata di Rende (CS)

gianfelice@mat.unical.it

\item Filippo MAIMONE and Vinicio PELINO

Italian Air Force, CNMCA

Aeroporto "De Bernardi"

Via di Pratica di Mare

I-00040 Roma

maimone@meteoam.it

pelino@meteoam.it

\item Sandro VAIENTI

UMR-6207 Centre de Physique Th\'{e}orique, CNRS, Universit\'{e}
d'Aix-Marseille I, II, Universit\'{e} du Sud, Toulon-Var and FRUMAM,
F\'{e}d\'{e}ration de Recherche des Unit\'{e}s des Math\'{e}matiques de Marseille

CPT Luminy, Case 907, F-13288 Marseille CEDEX 9

vaienti@cpt.univ-mrs.fr
\end{itemize}

\end{document}